\documentclass[11pt,a4paper,reqno]{amsart}%
\usepackage{amssymb}
\usepackage{amsmath}
\usepackage{amsfonts}
\usepackage{amsthm}
\usepackage{color}
\usepackage[usenames,dvipsnames]{xcolor}
\usepackage{bbm}
\usepackage{graphicx}
\usepackage{rotating}
\usepackage{hyperref}
\usepackage{mathrsfs}
\usepackage[bb=boondox]{mathalfa}
\usepackage{cmll}
\usepackage[all, cmtip]{xy}%
\providecommand{\U}[1]{\protect\rule{.1in}{.1in}}
\newtheorem{theorem}{Theorem}[subsection]

\newtheorem{lemma}[theorem]{Lemma}
\newtheorem{proposition}[theorem]{Proposition}
\newtheorem*{theorem*}{Theorem}

\theoremstyle{definition}
\newtheorem{definition}[theorem]{Definition}%[subsection]
\newtheorem{examplex}[theorem]{Example}
\newenvironment{example}
  {\pushQED{\qed}\examplex}
  {\popQED\endexamplex}
  
\theoremstyle{remark}
\newtheorem{remarkx}[theorem]{Remark}
\newenvironment{remark}
  {\pushQED{\qed}\remarkx}
  {\popQED\endremarkx}

\thanks{}
\numberwithin{equation}{section}

\makeatletter
\def\@tocline#1#2#3#4#5#6#7{\relax
  \ifnum #1>\c@tocdepth % then omit
  \else
    \par \addpenalty\@secpenalty\addvspace{#2}%
    \begingroup \hyphenpenalty\@M
    \@ifempty{#4}{%
      \@tempdima\csname r@tocindent\number#1\endcsname\relax
    }{%
      \@tempdima#4\relax
    }%
    \parindent\z@ \leftskip#3\relax \advance\leftskip\@tempdima\relax
    \rightskip\@pnumwidth plus4em \parfillskip-\@pnumwidth
    #5\leavevmode\hskip-\@tempdima
      \ifcase #1
       \or\or \hskip 1.5 em \or \hskip 2em \else \hskip 3em \fi%
      #6\nobreak\relax
    \hfill\hbox to\@pnumwidth{\@tocpagenum{#7}}\par% <---- \dotfill -> \hfill
    \nobreak
    \endgroup
  \fi}
\makeatother

\begin{document}
\newcommand{\R}{{\mathbb R}}
\newcommand{\C}{{\mathbb C}} 
\newcommand{\T}{{\mathbb T}}
\newcommand{\D}{{\mathbb D}}
\renewcommand{\P}{\mathbb P}

\newcommand{\Aa}{{\mathcal A}}
\newcommand{\Ii}{{\mathcal I}}
\newcommand{\Jj}{{\mathcal J}}
\newcommand{\Nn}{{\mathcal N}}
\newcommand{\Ll}{{\mathcal L}}
\newcommand{\Tt}{{\mathcal T}}
\newcommand{\Gg}{{\mathcal G}}
\newcommand{\Dd}{{\mathcal D}}
\newcommand{\Cc}{{\mathcal C}}
\newcommand{\Oo}{{\mathcal O}}

\newcommand{\pr}{\operatorname{pr}}
\newcommand{\bla}{\langle \! \langle}
\newcommand{\bra}{\rangle \! \rangle}
\newcommand{\blq}{[ \! [}
\newcommand{\brq}{] \! ]}
 \newcommand{\into}{\mathbin{\vrule width1.5ex height.4pt\vrule height1.5ex}}

\title[Holomorphic Jacobi Manifolds and Holomorphic Contact Groupoids]{Holomorphic Jacobi Manifolds \\ and Holomorphic Contact Groupoids}
\author{Luca Vitagliano}
\address{DipMat, Universit\`a degli Studi di Salerno, via Giovanni Paolo II n${}^\circ$ 123, 84084 Fisciano (SA) Italy.}
\email{lvitagliano@unisa.it}

\author{A\"issa Wade}
\address{Department of Mathematics, Penn State University, University Park, State College, PA 16802, USA.}
\email{wade@math.psu.edu}

\begin{abstract}
This paper belongs to a series of works aiming at exploring generalized (complex) geometry in odd dimensions. Holomorphic Jacobi manifolds were introduced and studied by the authors in a separate paper as special cases of \emph{generalized contact bundles}. In fact,  generalized contact bundles are nothing but  odd dimensional analogues of generalized complex manifolds.
%is the second part of a series of two papers dedicated to a systematic study of holomorphic Jacobi structures. In the first part, we introduced  and study the concept of a holomorphic Jacobi manifold  in a very natural way as well as various tools.
In the present paper, we solve the integration problem for holomorphic Jacobi manifolds by proving that they integrate to holomorphic contact groupoids.  A crucial tool in our proof is what we call the \emph{homogenization scheme}, which allows us to identify holomorphic Jacobi manifolds with homogeneous holomorphic Poisson manifolds and holomorphic contact groupoids with homogeneous complex symplectic groupoids.
\end{abstract}
\maketitle


\bigskip


\tableofcontents

\section{Introduction}

Generalized complex manifolds were introduced by Hitchin in  \cite{H2003} and they were unfolded  by Gualtieri in \cite{G2011}. Generalized complex manifolds must be even dimensional, so there arises a natural question:  is there an odd dimensional analogue? Several answers to this question have already been given \cite{IW2005,V2008,PW2011,W2012, S2015,AG2015}. Recently, elaborating on a previous work by Iglesias and the second author herself \cite{IW2005}, we proposed the notion of generalized contact bundle, which looks well motivated on conceptual grounds  \cite{VW2016}. 

There is a close relationship between generalized complex geometry and Poisson geometry. Namely, every generalized complex manifold is a Poisson manifold. Additionally, a holomorphic Poisson structure can be seen as a generalized complex structure (of special type) \cite{LSX2008}. Notice that the theory of  holomorphic Poisson structures has recently attracted a great deal of interest due to this connection (see, for instance, \cite{B2013, BX2015,  BCG2007, CGYS2015, H2012, LSX2008, SX2007}).

Similarly, every generalized contact bundle is a Jacobi bundle. Recall that a \emph{Jacobi bundle} is a line bundle with a Lie bracket on its sections, which is also a first order bi-differential operator.  In \cite{VW2016b} we also introduced holomorphic Jacobi manifolds, i.e.~complex manifolds $X$ equipped with a holomorphic line bundle $L \to X$ and a Lie bracket on the sheaf of holomorphic sections of $L$ which is also a first order bi-differential operator. Holomorphic Jacobi structures can be seen as generalized contact structures (of special type).

Holomorphic Jacobi manifolds can also be seen as homogeneous holomorphic Poisson-manifolds. Recall that a \emph{homogeneous holomorphic Poisson structure} on a complex manifold $X$ is a holomorphic Poisson bivector $\Pi$ together with a holomorphic vector field $H,$ called the homogeneity vector field, such that ${\mathcal L}_H \Pi = - \Pi$.  
 In \cite{VW2016b}, we showed that, as in the real case, given a complex manifold $X$ equipped with a holomorphic line bundle $L \to X$,  every holomorphic Jacobi structure $J$ on  
 $(X, L)$ determines a canonical homogeneous holomorphic Poisson structure on its slit complex dual $\widetilde{X}= L^* \smallsetminus 0$, called the \emph{Poissonization} of $J$. Here $0$ is the image of the zero section.
Conversely,  every homogeneous holomorphic Poisson manifold is the Poissonization of a canonical holomorphic Jacobi manifold around a non-singular point of the homogeneity vector. 

Basic examples of holomorphic Jacobi manifolds are provided by projective spaces over duals of complex Lie algebras and holomorphic contact manifolds. Holomorphic contact geometry is very different from real one, its study requires the use of different techniques \cite{Blair}, and the passage from the real to the complex case is not straightforward. For instance real contact manifolds are odd dimensional, while holomorphic contact manifolds have necessarily even real dimension. 

As we observed in \cite{VW2016b}, there is also a big difference between holomorphic Jacobi structures and holomorphic Poisson structures. Namely, as already mentioned, every holomorphic Poisson manifold is a generalized complex manifold. However, a holomorphic Jacobi structure on a complex manifold  $X=(M, j)$ gives rise to a generalized contact structure which is not on $M$ but rather on a circle bundle over $M$.  This observation generalizes the case of holomorphic contact manifolds which was considered by Kobayashi back in 1959 \cite{K1959}. Indeed, he noticed that given a holomorphic contact manifold  $X=(M, j)$,  there exists a principal $U(1)$-bundle over $M$  which is endowed with a natural real contact structure.

Various studies have been  conducted on the integration problem for real Jacobi structures (see \cite{CS2015, CZ2007} and references therein) and they show that \emph{Jacobi manifolds integrate to contact groupoids}. In this paper, we investigate the integration problem for holomorphic Jacobi  manifolds and show that they integrate to holomorphic contact groupoids. Recall that the analogous question for holomorphic Poisson manifolds has been answered in \cite{LSX2009}: \emph{holomorphic Poisson manifolds integrate to complex symplectic groupoids} \cite{LSX2009}.

 An important ingredient in our study is a general \emph{homogenization scheme}, i.e.~we establish a correspondence between calculus on a line bundle $L \to M$ and ``homogeneous calculus" on $\widetilde{M} = L^\ast \smallsetminus 0$. When applied to a Jacobi structure, the homogenization scheme produces its Poissonization. In \cite{CZ2007}, Crainic and Zhu proved that real Jacobi manifolds integrate to contact groupoids using the homogenization scheme but they only consider the case of trivial Jacobi bundles. Recently, Crainic and Salazar \cite{CS2015}, using their theory of Spencer operators \cite{CSS2015}, discuss the general case of Jacobi manifolds with generic Jacobi line bundles. But they do not use the homogenization scheme. Our strategy is a combination of those in \cite{CZ2007} and \cite{CS2015} (and \cite{LSX2009}) that we find particularly efficient. Namely, first we use Spencer operators to give a new proof for the following result established in \cite{LSX2009}: integrable Poisson-Nijenhuis manifolds are in one-to-one correspondence with source-simply connected symplectic-Nijenhuis groupoids. Using this result, we show that such a correspondence restricts to the smaller class of integrable homogeneous Poisson-Nijenhuis manifolds. Namely, integrable homogeneous Poisson-Nijenhuis manifolds are in one-to-one correspondence with source-simply connected homogeneous symplectic-Nijenhuis groupoids. Finally, we use this and the homogenization scheme to prove our main result (see Theorem \ref{main thm} for a precise statement):
 
 \begin{theorem*}
Integrable holomorphic Jacobi manifolds are in one-to-one correspondence with source-connected and simply connected holomorphic contact groupoids, i.e.~complex groupoids equipped with a multiplicative holomorphic contact structure.
 \end{theorem*}
 
Along the way we provide a new proof of Crainic and Salazar result. A key ingredient throughout the paper is the notion of a multiplicative Atiyah tensor. Multiplicative structures on Lie groupoids were investigated by many researchers (see, e.g., \cite{BD2017} and references therein). The natural framework for multiplicative Atiyah tensors is the setting on VB-groupoids \cite{M2005}, their derivations \cite{ETV2016} and $1$-jets. Multiplicative Atiyah tensors correspond to (ordinary) multiplicative tensor under the homogenization scheme. For instance, a multiplicative contact structure can be seen as a multiplicative symplectic Atiyah form, which, in its turn correspond, under the homogenization scheme, to a homogeneous multiplicative symplectic form. In this way we study Lie groupoids equipped with multiplicative Atiyah tensors, by studying Lie groupoids equipped with ordinary multiplicative tensors which are additionally homogeneous in a suitable sense.
 
  The paper contains three main sections. In Section 2,  we review the homogenization scheme and study the homogenization of multiplicative Atiyah tensors on line bundle groupoids. In Section 3, we  review known results about the integration of Poisson manifolds and extend them to homogeneous Poisson manifolds.  In particular, we establish that (homogeneous) Poisson-Nijenhuis manifolds integrate to (homogeneous) symplectic-Nijenhuis groupoids by taking the approach of Spencer operators. Section 4 contains our main result (Theorem \ref{main thm}):  integrable holomorphic Jacobi manifolds integrate to holomorphic contact groupoids. 
 
We assume the reader is familiar with Lie groupoids, Lie algebroids, and their Lie theory. Our main reference about this material is \cite{CF2011}.

\section{The homogenization scheme}\label{Sec:hom_trick}
Recall that there is a dictionary that allows to pass from symplectic to contact geometry, and from Poisson to Jacobi geometry, in a very natural way (see for example \cite{BGG2015, V2015}). This dictionary is ultimately based on what we call the \emph{homogenization scheme}. This consists in applying in a systematic way the equivalence between the categories of line bundles  and principal $\mathbb R^\times$-bundles. Here $\mathbb R^\times$ stands for the multiplicative group of non-zero reals. In this section we carefully review this equivalence and its implications for calculus on line bundles. Let us begin with line bundle maps.

Let $L_N \to N$ and $L \to M$ be line bundles, and let $\phi_L : L_N \to L$ be a vector bundle map covering a smooth map $\phi : M \to N$. In what follows by a line bundle map we will always mean a \emph{regular} vector bundle map $\phi_L : L_N \to L$, i.e.~a vector bundle map which is an isomorphism on fibers. In particular, $\phi_L$ induces an isomorphism $L_N \simeq \phi^\ast L$, which we will often understand if this does not lead to confusion. Sections of vector bundles can be pulled-back along regular vector bundle maps. In particular, sections of a line bundle $L$ can be pulled-back along a line bundle map $\phi_L : L_N \to L$. Namely, for $\lambda \in \Gamma (L)$ we denote by $\phi_L^\ast \lambda$ (or simply $\phi^\ast \lambda$ is this does not lead to confusion) the section of $L_N$ defined by
\[
(\phi^\ast_L \lambda)_y := \phi_L |_{(L_N)_y}^{-1} (\lambda_{\phi (y)}).
\]
This agrees with the standard pull-back construction under the identification $L_N \simeq \phi^\ast L$.

	\subsection{Line bundles and principal $\mathbb R^\times$-bundles} 

Let $L \to M$ be a line bundle, and let $\widetilde M = L^\ast \smallsetminus 0$ be its slit dual bundle. We will usually denote by $p : \widetilde M \to M$ the projection. Clearly, $\widetilde M$ is a principal $\mathbb R^\times$-bundle over $M$ and every principal $\mathbb R^\times$-bundle is of this form. Denote by
\[
h : \mathbb R^\times \times \widetilde M \to \widetilde M, \quad (r, \epsilon) \mapsto h_r \epsilon := r \cdot \epsilon
\]
the action of $\mathbb R^\times$ on $\widetilde M$. We will often call $\widetilde M$ the \emph{homogenization} of $L$. The reason for this terminology will be more clear later on in this section.

If $L_N \to N$ and $L \to M$ are two line bundles and $\phi_L : L_N \to L$ is a line bundle map covering a smooth map $\phi : N \to M$, then there is an obvious line bundle map $\phi^\star : L^\ast_N \to L^\ast$, also covering $\phi$, defined on fibers as the inverse of the transpose of $\phi_L$. We denote by $\widetilde \phi : \widetilde N \to \widetilde M$ the restriction of $\phi^\star$ to the homogenization. The map $\widetilde \phi$ is $\mathbb R^\times$-equivariant and every $\mathbb R^\times$-equivariant map $\widetilde N \to \widetilde M$ arises in this way. This shows that correspondences $L \mapsto \widetilde M$, and $\phi_L \mapsto \widetilde \phi$ establish an equivalence between the category of line bundles and line bundle maps and the category of principal $\mathbb R^\times$-bundles and equivariant maps. This equivalence determines a correspondence between calculus on a line bundle $L \to M$ and calculus on its homogenization $\widetilde M$ preserving the $\mathbb R^\times$-equivariance in a suitable sense. We illustrate this phenomenon in the remaining part of this section. We begin showing that sections of $L$ correspond to certain functions on $\widetilde M$.

First of all, denote by $Z$ the restriction to $\widetilde M$ of the Euler vector field on $L^\ast$. So $Z$ is the fundamental vector field corresponding to the canonical generator $1$ in the abelian Lie algebra $\mathbb R$ of $\mathbb R^\times$. We will keep calling it the \emph{Euler vector field}. Let $f \in C^\infty (\widetilde M)$. The following two conditions are equivalent
\begin{itemize}
\item[$\triangleright$] $Z(f) = f$, and $h_{-1}^\ast f = -f$,
\item[$\triangleright$] $h_r^\ast f = r \cdot f$ for all $r \in \mathbb R^\times$.
\end{itemize} 
A function satisfying one, hence both, of the above two conditions is called a \emph{homogeneous function} (of degree $1$). Every section $\lambda \in \Gamma (L)$ can be seen as a fiber-wise linear function on $L^\ast$, and, by restriction, as a homogeneous function, denoted $\widetilde \lambda$ on $\widetilde M$. Correspondence $\lambda \mapsto \widetilde \lambda$ is one-to-one. Function $\widetilde \lambda$ will be often called the \emph{homogenization} of section $\lambda$. There is a more geometric description of this construction. Namely, let 
\[
\mathbb R_{\widetilde M} = \widetilde M \times \mathbb R \to \widetilde M
\]
be the trivial line bundle over $\widetilde M$. Then, there is a canonical line bundle map $p_{L} : \mathbb R_{\widetilde M} \to L$ covering $p : \widetilde M \to M$:
\[
\begin{array}{c}
\xymatrix{ \mathbb R_{\widetilde M} \ar[r]^-{p_L} \ar[d] & L \ar[d] \\
 \widetilde M \ar[r]^-p & M}
 \end{array}.
\]
Line bundle map $p_L$ is defined as follows. Let $(\epsilon, r) \in \widetilde M \times \mathbb R$. Then $p_L (\epsilon, r) \in L_{p(\epsilon)}$ is implicitly given by 
\[
\langle \epsilon, p_L(\epsilon, r) \rangle = r,
\]
where $\langle -,- \rangle : L^\ast \otimes L \to \mathbb R$ is the duality pairing. It is now easy to see that, for every $\lambda \in \Gamma (L)$, its homogenization $\widetilde \lambda$, regarded as a section of $\mathbb R_{\widetilde M}$, is precisely the pull-back $p_L^\ast \lambda$:
\[
\widetilde \lambda = p_L^\ast \lambda.
\]

	\subsection{Atiyah tensors and their homogenization}\label{Sec:hom_Atiyah}
	
Not only sections of $L$ correspond to homogeneous functions on $\widetilde M$, actually all (first order) calculus on $L$ correspond to a \emph{homogeneous calculus} on $\widetilde M$. To see this, notice that the building blocks of first order calculus on $L$ are \emph{derivations}  and $1$-jets of $L$. Recall that a derivation of a vector bundle $E \to M$ is an $\mathbb R$-linear operator
\[
\Delta : \Gamma (E) \to \Gamma (E)
\]
satisfying the following \emph{Leibniz rule}
\[
\Delta (f \varepsilon) = X(f) \varepsilon + f \Delta \varepsilon, \quad f \in C^\infty (M), \quad \varepsilon \in \Gamma (E),
\]
for a, necessarily unique, vector field $X \in \mathfrak X (M)$, called the \emph{symbol} of $\Delta$ and denoted by $\sigma (\Delta)$. Derivations are sections of a Lie algebroid $DE \rightarrow M$, called the \emph{gauge algebroid}, or, sometimes, the \emph{Atiyah algebroid}, of $E$, whose anchor is the symbol, and whose bracket is the commutator of derivations. A point in $DE$ over a point $x \in M$ is an $\mathbb R$-linear map $\Delta : \Gamma (E) \to E_x$ satisfying the Leibniz rule $\Delta (f \varepsilon) = v(f) \varepsilon_x + f(x) \Delta \varepsilon$ for some tangent vector $v \in T_x M$, the \emph{symbol} of $\Delta$.

\begin{remark}
We stress for later purposes that correspondence $E \mapsto DE$ is functorial, in the sense that every regular vector bundle map $\phi_E : E_N \to E$ between vector bundles $E_N \to N$ and $E \to M$ induces a (generically non-regular) vector bundle map $D \phi_E : D E_N \to D E$ via formula
\[
D \phi_E (\Delta) \varepsilon = \phi_E  \Delta ( \phi_L^\ast \varepsilon)
\]
for all $\Delta \in D E_N$, and $\varepsilon \in \Gamma (E)$. The correspondence $\phi_E \mapsto D \phi_E$ preserves identity and compositions. When there is no risk of confusion, we write $D \phi$ instead of $D \phi_E$.
\end{remark}

Derivations of a vector bundle $E$ can be seen as \emph{linear vector fields} on $E$, i.e.~vector fields generating a flow by vector bundle automorphisms. Namely, for every derivation $\Delta$ of $E$, there exists a unique flow $\{ \phi_t \}$ by vector bundle automorphisms $\phi_t : E \to E$ such that 
\[
\Delta \varepsilon =  \frac{d}{dt} |_{t = 0} \phi_t^\ast \varepsilon
\]
for all $\varepsilon \in \Gamma (E)$. So $\Delta$ corresponds to a linear vector field on $E$ and this correspondence is one-to-one. 

 In the case of  a line bundle $L \to M$, every first order differential operator $\Gamma (L) \to \Gamma (L)$ is a derivation. Consequently there is a vector bundle isomorphism $DL \simeq \operatorname{Hom} (\mathfrak J^1 L, L)$, and a non-degenerate pairing $\langle -, - \rangle : \mathfrak J^1 L \otimes DL  \to L$, where $\mathfrak J^1 L \to M$ is the first jet bundle of $L$.

\begin{definition}
An \emph{Atiyah $(l, m)$-tensor} on $L$ is a section of vector bundle
\[
(DL)^\ast{}^{\otimes l} \otimes (\mathfrak J^1 L)^\ast{}^{\otimes m} \otimes L
\]
or, equivalently, a vector bundle map
\[
(DL)^{\otimes l} \otimes (\mathfrak J^1 L)^{\otimes m} \to L,
\]
where $m,l \in \mathbb N_0$.
\end{definition}

\begin{example}
Of special importance are skew-symmetric Atiyah $(l, 0)$-tensors. They can be seen as $l$-cochains in the de Rham complex of the gauge algebroid $DL$ with coefficients in its tautological representation $L$. Similarly, skew-symmetric Atiyah $(0,m)$-tensors are skew-symmetric $m$-multiderivations of $L$. See Examples \ref{ex:ext_diff} and \ref{ex:Schout_brack} for more details.
\end{example}

We want to show that Atiyah $(l,m)$-tensors on $L$ correspond to certain tensors on $\widetilde M$. First of all, let $\mathcal T$ be an $(l,m)$ tensor on $\widetilde M$, i.e.~a section of the tensor bundle
\[
(T^\ast \widetilde M)^{\otimes l} \otimes (T\widetilde M)^{\otimes m},
\]
or, equivalently, a vector bundle map
\[
(T\widetilde M)^{\otimes l} \otimes (T^\ast \widetilde M)^{\otimes m} \to \mathbb R_{\widetilde M},
\]

\begin{proposition}\label{prop:homo}
The following two conditions are equivalent
\begin{itemize}
\item[$\triangleright$] $\mathcal L_Z \mathcal T = (1 - m) \mathcal T$, and $h_{-1}^\ast \mathcal T = (-)^{1-m} \mathcal T$,
\item[$\triangleright$] $h_r^\ast \mathcal T = r^{1-m} \mathcal T$ for all $r \in \mathbb R^\times$.
\end{itemize}
\end{proposition}

The proof is straightforward and it is left to the reader. 

\begin{definition}\label{def:homo}
An $(l,m)$-tensor $\mathcal T$ on $\widetilde M$ is \emph{homogeneous} if it satisfies one, hence both, of the equivalent conditions in Proposition \ref{prop:homo}.
\end{definition}

At a first glance, Definition \ref{def:homo} may seem surprising. But it's main motivation is the following:

\begin{theorem}\label{theor:homo}
There is a one-to-one correspondence between Atiyah $(l,m)$-tensors on $L$ and homogeneous $(l,m)$-tensors on $\widetilde M$. Geometrically, it can be described as follows: tensor bundle $(T^\ast \widetilde M)^{\otimes l} \otimes (T\widetilde M)^{\otimes m}$ is canonically isomorphic to the pull-back bundle
\[
p^\ast \left( (DL)^\ast{}^{\otimes l} \otimes (\mathfrak J^1 L)^\ast{}^{\otimes m} \otimes L \right)
\]
and, under this isomorphism, homogeneous $(l,m)$-tensors correspond to pull-back sections.
\end{theorem}

\begin{proof}
We only sketch the proof, leaving the simple details to the reader. Atiyah $(0,0)$-tensors are just sections of $L$. We already know that they correspond to homogeneous $(0,0)$-tensors, i.e.~functions on $\widetilde M$. Continue with Atiyah $(0,1)$-tensors. They are the same as derivations of $L$. Now, a vector field on $\widetilde M$ is completely determined by its action on homogeneous functions. It follows that, for every derivation $\Delta$ there exists a unique vector field $\widetilde \Delta$ on $\widetilde M$ such that
\[
\widetilde \Delta (\widetilde \lambda) = \widetilde{ \Delta \lambda}
\]
for all sections $\lambda \in \Gamma (L)$. Additionally, $\widetilde \Delta$ is homogeneous. Indeed, for any homogeneous function $f$ on $\widetilde M$ we have
\[
(\mathcal L_{\mathcal E} \widetilde \Delta )(f) = \mathcal E \widetilde \Delta f - \widetilde \Delta \mathcal E f = \widetilde \Delta f - \widetilde \Delta f = 0,
\]
where we used that both $f$ and $\widetilde \Delta f$ are homogeneous functions. It follows that $\mathcal L_{\mathcal E} \widetilde \Delta = 0$, as claimed.

Next, we consider Atiyah $(1,0)$-tensors. They are the same as sections of $\mathfrak J^1 L$. As differential forms on $\widetilde M$ are completely determined by contraction with homogeneous vector fields, for every section $\psi$ of $\mathfrak J^1 L$, there exists a unique $1$-form $\widetilde \psi$ on $\widetilde M$ such that
\[
\langle \widetilde \psi, \widetilde \Delta \rangle = \widetilde{\langle \psi, \Delta \rangle}
\]
for all derivations $\Delta$ of $L$. 
One can see that $\widetilde \psi$ is homogeneous in a similar way as for derivations. It is now clear how to proceed: let $T$ be any Atiyah $(l,m)$-tensor, and interpret it as a vector bundle map 
\[
T : (DL)^{\otimes l} \otimes (\mathfrak J^1 L)^{\otimes m} \to L,
\]
then there exists a unique homogeneous $(l,m)$-tensor $\widetilde T$ on $\widetilde M$ such that
\begin{equation}\label{eq:T_tilde}
\widetilde T (\widetilde \Delta_1, \ldots, \widetilde \Delta_l, \widetilde \psi_1, \ldots, \widetilde \psi_m) = \widetilde{T(\Delta_1, \ldots, \Delta_l, \psi_1, \ldots, \psi_m)}
\end{equation}
for all derivations $\Delta_1, \ldots, \Delta_l$ of $L$, and all sections $\psi_1, \ldots, \psi_m$ of $\mathfrak J^1 L$.

For the second part of the statement, begin defining a vector bundle map $p_T : T \widetilde M \to DL$ covering $\widetilde M \to M$ as follows (see also \cite[Appendix]{V2015}). Let $\epsilon \in \widetilde M$. We map a tangent vector $v \in T_{\epsilon} \widetilde M$ to derivation $p_T (v) \in D_{p(\epsilon)} L$ implicitly defined by
\[
\langle \epsilon, p_T (v) \lambda \rangle = v \widetilde \lambda,
\]
for all $\lambda \in \Gamma (L)$. Clearly, $p_T$ is fiber-wise injective. For dimension reasons it is also fiber-wise surjective, so diagram
\[
\xymatrix{
T \widetilde M \ar[r]^-{p_T} \ar[d] & DL \ar[d] \\
\widetilde M \ar[r]^-p & M
}
\]
is a pull-back diagram inducing an isomorphism $T \widetilde M \simeq p^\ast DL$. 

Next define a vector bundle map $p_{T^\ast} : T^\ast \widetilde M \to \mathfrak J^1 L$ as follows. Let $\epsilon \in \widetilde M$, and notice that any covector $\theta \in T_\epsilon^\ast \widetilde M$ is of the form $d_{\epsilon} \widetilde \lambda$ for some homogeneous function $\widetilde \lambda$. Now put
\[
p_{T^\ast} (d_\epsilon \widetilde \lambda ) := \mathfrak j^1_{p(\epsilon)} \lambda.
\]
Equivalently, for any $\theta \in T^\ast_\epsilon \widetilde M$ and any $\Delta \in D_{p(\epsilon)}L$,
\[
\langle \epsilon, p_{T^\ast} (\theta) (\Delta) \rangle =  \theta (v),
\]
where $v \in T_\epsilon \widetilde M$ is any tangent vector such that $p_T (v) = \Delta$. This shows that $p_{T^\ast}$ is a well-defined vector bundle map and that diagram
\[
\xymatrix{
T^\ast \widetilde M \ar[r]^-{p_{T^\ast}} \ar[d] & \mathfrak J^1L \ar[d] \\
\widetilde M \ar[r]^-p & M
}
\]
is a pull-back diagram, so that $T^\ast \widetilde M \simeq p^\ast \mathfrak J^1 L$.

Finally, we get isomorphisms

\begin{equation}\label{eq:pull-back}
\begin{aligned}
& p^\ast \left( (DL)^\ast{}^{\otimes l} \otimes (\mathfrak J^1 L)^\ast{}^{\otimes m} \otimes L \right)  \\
& \simeq 
p^\ast (DL)^\ast{}^{\otimes l} \otimes p^\ast  (\mathfrak J^1 L)^\ast{}^{\otimes m} \otimes p^\ast  L   \\
& \simeq (T\widetilde M)^\ast{}^{\otimes l} \otimes (T^\ast \widetilde M)^\ast{}^{\otimes m} \otimes \mathbb R_{\widetilde M} \\
& \simeq (T^\ast \widetilde M)^{\otimes l} \otimes (T \widetilde M)^{\otimes m}
\end{aligned}
\end{equation}

We leave it to the reader to check that, under this isomorphism, homogeneous tensors correspond to pull-back sections.
\end{proof}

\begin{remark}
The proof of Theorem \ref{theor:homo} shows that (higher degree) homogeneous tensors are characterized by the property that contraction with homogeneous vector fields and $1$-forms gives a homogeneous function.
\end{remark}

\begin{remark}
It follows from isomorphism (\ref{eq:pull-back}) that there are natural vector bundle maps
\begin{equation}\label{eq:p_l,m}
p_{l,m} : (T^\ast \widetilde M)^{\otimes l} \otimes (T \widetilde M)^{\otimes m} \to (DL)^\ast{}^{\otimes l} \otimes (\mathfrak J^1 L)^\ast{}^{\otimes m} \otimes L  
\end{equation}
covering $p$. Explicitly they are defined as follows. Let $\epsilon \in \widetilde M$, and $\mathcal T \in (T_\epsilon^\ast \widetilde M)^{\otimes l} \otimes (T_\epsilon \widetilde M)^{\otimes m}$. Then $p_{l,m} (\mathcal T)$ is uniquely determined by
\[
\langle \epsilon, p_{l,m} (\mathcal T) (\Delta_1, \ldots, \Delta_l; \psi_1, \ldots, \psi_m)) \rangle = \mathcal T (v_1, \ldots, v_l; \theta_1, \ldots \theta_m)
\]
for all $\Delta_i \in D_{p(\epsilon)} L$, and all $\psi_j \in \mathfrak J^1_{p(\epsilon)} L$, where $v_i \in T_\epsilon \widetilde M$, and $\theta_j \in T^\ast_\epsilon \widetilde M$, are such that $\Delta_i = p_T (v_i)$, and $\psi_j = p_{T^\ast} (\theta_j)$.

The $p_{l,m}$ are all principal $\mathbb R^\times$-bundle projections. Let us describe the corresponding Euler vector field. First of all the Lie derivative $\mathcal L_Z$ of tensors along the Euler vector field $Z$ is a derivation of the vector bundle 
\[
(T^\ast \widetilde M)^{\otimes l} \otimes (T \widetilde M)^{\otimes m} \to \widetilde M.
\]
It easily follows from (the second part of) Theorem \ref{theor:homo} that the Euler vector field of the principal $\mathbb R^\times$-bundle (\ref{eq:p_l,m}) agrees with the linear vector field corresponding to derivation
\[
\mathcal L_Z + (m-1) \mathbb 1
\]
where $\mathbb 1$ is the identity derivation. 
\end{remark}

\begin{remark}
Beware that vector bundle $(T^\ast \widetilde M)^{\otimes l} \otimes (T \widetilde M)^{\otimes m} $ is also canonically isomorphic to the pull-back along $p$ of 
$
(DL)^\ast{}^{\otimes l} \otimes (\mathfrak J^1 L)^\ast{}^{\otimes m} 
$
(without the last factor $L$). Accordingly, there is another (free and proper) action of $\mathbb R^\times$ on $(T^\ast \widetilde M)^{\otimes l} \otimes (T \widetilde M)^{\otimes m} $. It can be checked that the latter is the natural lift to tensors of the $\mathbb R^\times$-action on $\widetilde M$. The corresponding Euler vector field is the linear vector field corresponding to the Lie derivative of tensors along $Z$. This is also useful in several situations. However we will not need it in this paper, so we will not insist in this direction.
\end{remark}

\begin{example}
The Euler vector field $Z$ is homogeneous itself. So it corresponds to a derivation $\Delta$ of $L$. Actually $\Delta$ is just \emph{minus} the identity $\mathbb 1 : \Gamma (L) \to \Gamma (L)$.
\end{example}

    \subsection{Some distinguished natural operations on Atiyah tensors} 
    
Denote by $T^{l,m} (\widetilde M)$ the space of $(l,m)$-tensors on $\widetilde M$, and by $T^{l,m}_L$ the space of Atiyah $(l,m)$-tensors on the line bundle $L \to M$. Take an $\mathbb R$-multilinear map
    \[
   \widetilde \mu : T^{l_1,m_1} (\widetilde M) \otimes_{\mathbb R} \cdots \otimes_{\mathbb R} T^{l_k, m_k} (\widetilde M) \to T^{l,m} (\widetilde M).
    \]
If $\widetilde \mu$ preserves homogeneous tensors, then it induces a map 
\[
\mu : T^{l_1,m_1}_L \otimes_{\mathbb R} \cdots \otimes_{\mathbb R} T^{l_k, m_k}_L \to T^{l,m}_L.
\]
There are several natural operations on Atiyah tensors arising in this way.  We now present a list of those operations that will be relevant in this paper.

\begin{example}[Contraction with a tensor valued vector field]
Contraction of one contravariant index with the first covariant index determines an $\mathbb R$-bilinear (actually $C^\infty (\widetilde M)$-bilinear) map
\begin{equation}\label{eq:contr}
T^{l_1, 1} (\widetilde M) \otimes T^{l_2, m_2} (\widetilde M) \to T^{l_1 + l_2 - 1, m_2} (\widetilde M), \quad (T, U) \mapsto T \into U.
\end{equation}
Clearly, contraction $\into$ preserves homogeneous tensors. So it induces a map
\[
T^{l_1, 1} _L \otimes T^{l_2, m_2}_L \to T^{l_1 + l_2 - 1, m_2}_L, \quad (T, U) \mapsto T \into U
\]
that can be described as follows. First notice that
\[
T^{l_1, 1}_L = (DL)^\ast{}^{\otimes l_1} \otimes (\mathfrak J^1 L)^\ast \otimes L \simeq (DL)^\ast{}^{\otimes l_1} \otimes DL
\]
So $T \in T^{l_1, 1}_L$ can be seen as a vector bundle map
\[
T : (DL)^{\otimes l_1} \to DL.
\]
If we see $U \in T^{l_2, m_2}_L$ as a vector bundle map
\[
U : (DL)^\ast{}^{\otimes l_2} \otimes (\mathfrak J^1 L)^\ast{}^{\otimes m_2} \to L
\]
Then $T \into U$ is the vector bundle map
\[
T \into U :  (DL)^\ast{}^{\otimes l_1 + l_2 -1} \otimes (\mathfrak J^1 L)^\ast{}^{\otimes m_2} \to L
\]
given by
\[
\begin{aligned}
& (T \into U)(\Delta_1, \ldots, \Delta_{l_1 +l_2 -1}, \psi_1, \ldots, \psi_{m_2}) \\
& = U (T (\Delta_1, \ldots, \Delta_{l_1}), \Delta_{l_1 +1}, \ldots, \Delta_{l_1 + l_2 -1} ; \psi_1, \ldots, \psi_{m_2}),
\end{aligned}
\]
for all derivations $\Delta_i $ of $L$ and all sections $\psi_j$ of $\mathfrak J^1 L$. By definition
\[
\widetilde{T \into U} = \widetilde T \into \widetilde U
\]
for all $T \in T^{l_1, 1}_L$ and $U \in T^{l_2, m_2}_L$. In the following, we will only consider the case $l_1 = l_2 = m_2 = 1$. In this case, contraction (\ref{eq:contr}) is just composition of $(1,1)$-tensors seen as endomorphisms $T \widetilde M \to T \widetilde M$ of the tangent bundle. At the level of Atiyah tensors, we get composition of Atiyah $(1,1)$-tensors seen as endomorphisms $DL \to DL$ of the gauge algebroid. 
\end{example}  

\begin{example}[Exterior differential]\label{ex:ext_diff}
We call \emph{Atiyah} forms, the skew-symmetric Atiyah $(\bullet ,0)$-tensors, and denote them by $\Omega^\bullet_L$. They can be seen as vector bundle maps
\[
\wedge^\bullet DL \to L,
\]
and correspond to differential forms $\Omega^\bullet (\widetilde M)$ under homogenization. The exterior differential
\[
d : \Omega^{\bullet} (\widetilde M) \to \Omega^{\bullet + 1} (\widetilde M)
\]
preserves homogeneous forms. So it induces a differential
\[
d_D : \Omega^\bullet_L \to \Omega^{\bullet + 1}_L.
\]
It is easy to see that $d_D$ agrees with the de Rham differential of the gauge algebroid with coefficient in its tautological representation $L$. We have
\[
\widetilde{d_D \omega} = d \widetilde \omega
\]
for all $\omega \in \Omega^\bullet_L$. On Atiyah $0$-forms $\Omega^0_L = \Gamma (L)$, differential $d_D$ agrees with the jet prolongation $\mathfrak j^1 : \Gamma (L) \to \Gamma (\mathfrak J^1 L) = \Omega^1_L$.
\end{example}

\begin{example}[Schouten bracket]\label{ex:Schout_brack}
A skew-symmetric Atiyah $(0, m)$-tensor $\Delta$ can also be seen as a skew-symmetric $m$-multiderivation of $L$, i.e.~a skew-symmetric, $\mathbb R$-multilinear, $m$-ary multibracket
\[
\{-, \ldots, -\} : \Gamma (L) \times \cdots \times \Gamma (L) \to \Gamma (L)
\]
which is a derivation in each argument. To see this, first interpret $\Delta$ as a vector bundle map
\[
\Delta : \wedge^m \mathfrak J^1 L \to L
\]
and then put
\[
\{ \lambda_1, \ldots, \lambda_m \} := \Delta (\mathfrak j^1 \lambda_1, \ldots, \mathfrak j^1 \lambda_m),
\]
$\lambda_1, \ldots, \lambda_m \in \Gamma (L)$. Skew-symmetric Atiyah $(0, \bullet)$-tensors will be denoted by $\mathcal D^\bullet L$. They correspond to multivectors $\mathfrak X^\bullet (\widetilde M)$ under homogenization. The Schouten bracket
\[
[-,-]^{S} : \mathfrak X^\bullet (\widetilde M) \otimes_{\mathbb R} \mathfrak X^\bullet (\widetilde M) \to \mathfrak X^\bullet (\widetilde M)
\]
preserves homogeneous multivectors. So it induces a bracket
\[
[-,-]^{SJ} : \mathcal D^\bullet L \otimes_{\mathbb R} \mathcal D^\bullet L \to  \mathcal D^\bullet L
\]
sometimes called the \emph{Schouten-Jacobi bracket} (see, e.g., \cite{LOTV2014, LTV2016}). Let $\Delta_1 \in  \mathcal D^{m_1} L$ and $\Delta_2 \in  \mathcal D^{m_2} L$, and denote by $\{-, \ldots, -\}_{[1]}$ and $\{-,\ldots, -\}_{[2]}$ the corresponding multiderivations. Then the Schouten-Jacobi bracket $[\Delta_1, \Delta_2]^{SJ}$ corresponds to $(m_1 + m_2 -1)$-multiderivation $\{-, \ldots, -\}$ given by the following \emph{Gerstenhaber formula}:
\begin{equation}\label{eq:SJ}
\begin{aligned}
& \{ \lambda_1, \ldots, \lambda_{m_1+m_2 -1}\} \\
& = \sum_{\sigma \in S_{m_2, m_1 -1}} (-)^\sigma \{ \{\lambda_{\sigma(1)}, \ldots, \lambda_{\sigma (m_2)} \}_{[2]}, \lambda_{\sigma(m_2 +1)}, \ldots, \lambda_{\sigma (m_1 + m_2 -1)} \}_{[1]}
\\
& \quad - (-)^{(m_1 +1)(m_2 +1)} \sum_{\sigma \in S_{m_1, m_2 -1}} (-)^\sigma \{ \{\lambda_{\sigma(1)}, \ldots, \lambda_{\sigma (m_1)} \}_{[1]}, \lambda_{\sigma(m_1 +1)}, \ldots, \lambda_{\sigma (m_1 + m_2 -1)} \}_{[2]}.
\end{aligned}
\end{equation}
We have
\[
\widetilde{[\Delta_1, \Delta_2]}{}^{SJ} = [\widetilde \Delta_1, \widetilde \Delta_2]^S.
\]
Notice that the Schouten-Jacobi bracket can be defined on skew-symmetric multiderivations of any (non-necessarily rank $1$) vector bundle by the same fomula (\ref{eq:SJ}).
\end{example}

\begin{example}[Frolicher-Nijenhuis bracket]
Skew-symmetric Atiyah $(\bullet, 1)$-tensors can be seen as vector bundle maps
\[
U : \wedge^\bullet DL \to DL.
\]
Denote them by $\Omega^\bullet_L (DL)$. They correspond to vector valued differential forms (equivalently, form valued vector fields) $\Omega^\bullet (\widetilde M, T \widetilde M)$ under homogenization. The Frolicher-Nijenhuis bracket
\[
[-,-]^{FN} : \Omega^\bullet (\widetilde M, T \widetilde M) \otimes_{\mathbb R} \Omega^\bullet (\widetilde M, T \widetilde M) \to \Omega^\bullet (\widetilde M, T \widetilde M)
\]
preserves homogeneity. So it induces a bracket
\[
[-,-]^{FN}_D : \Omega^\bullet_L (DL) \otimes_{\mathbb R} \Omega^\bullet_L (DL) \to \Omega^\bullet_L (DL).
\]
Notice that $\Omega^1_L (DL)$ consists of Atiyah $(1,1)$-tensors. For $U, V \in \Omega^1_L (DL)$, their bracket $[U, V]^{FN}_D \in \Omega^2_L (DL)$ is given by formula
\[
\begin{aligned}
& {[U, V]}^{FN}_D (\Delta_1, \Delta_2) \\
& = [U(\Delta_1), V(\Delta_2)] - U [V(\Delta_1), \Delta_2] - U[ \Delta_1, V(\Delta_2)] 
+ UV[\Delta_1, \Delta_2] + (U \leftrightarrow V),
\end{aligned}
\]
for all derivations $\Delta_1, \Delta_2$ of $L$. In particular, $[U,U]^{FN}_D$ is twice the \emph{Lie algebroid Nijenhuis torsion} $\mathcal T_U$ of $U : DL \to DL$:
\[
\mathcal T_U (\Delta_1, \Delta_2) = [U(\Delta_1), U(\Delta_2)] - U [U(\Delta_1), \Delta_2] - U[ \Delta_1, U(\Delta_2)] 
+ U^2[\Delta_1, \Delta_2].
\]
 We have
\begin{equation}\label{eq:hom_Frol}
\widetilde{[U, V]}{}^{FN}_D = [\widetilde U, \widetilde V]^{FN}.
\end{equation}
\end{example}

    \subsection{Examples}
    
    In this section we discuss more closely the homogenization of those specific classes of Atiyah tensors that will be relevant in the bulk of the paper.
    
    \begin{example}[Contact structures]\label{ex:contact}
Fix a line bundle $L \to M$, and consider a contact structure $H$ on $M$ with the additional property that the normal bundle $TM/H$ is isomorphic to $L$. We also fix once for all an isomorphism $TM/H \to L$. In this situation, $H$ is actually equivalent to a \emph{symplectic Atiyah form} on $L$, i.e.~a non degenerate, $d_D$-closed, Atiyah $2$-form $\omega$ \cite{V2015}. Correspondence $H \mapsto \omega$ can be defined in this way. Composing the projection $TM \to TM/H$ with the isomorphism $TM/H \to L$, we get an $L$-valued $1$-form $\theta$ on $M$:
\[
\theta : TM \to L
\]
with the property $\ker \theta = H$. The form $\theta$ will be sometimes referred to as the \emph{structure form} of $H$. Further composing $\theta$ with the symbol $\sigma : DL \to L$, we get an Atiyah $1$-form $\Theta$. Put $\omega = d_D \Theta$. It can be checked that $\omega$ is a symplectic Atiyah $2$-form in the above sense. Conversely, let $\omega \in \Omega^2_L$ be a symplectic Atiyah $2$-form. Contraction $i_{\mathbb 1} \omega$ descends to an $L$-valued $1$-form $\theta$ on $M$, i.e.~there is a unique $\theta \in \Omega^1 (M, L)$ such that
\[
i_{\mathbb 1} \omega = \theta \circ \sigma.
\]
Put $H = \ker \theta$. It can be checked that $H$ is a contact distribution and this construction inverts the previous one. 

Now let $H$ be a contact structure and let $\omega$ be the associated symplectic Atiyah form. The homogenization $\widetilde \omega$ of $\omega$ is a homogeneous symplectic form on $\widetilde M$, i.e.~$\mathcal L_Z \widetilde \omega = \widetilde \omega$. Triple $(\widetilde M, \widetilde \omega, Z)$ agrees with the standard \emph{symplectization} of the contact manifold $(M, H)$. 
    \end{example}
    
     \begin{example}[Jacobi structures]\label{ex:Jacobi_structures}
   This example generalizes previous one. Recall that a \emph{Jacobi structure} on a manifold $M$ is a line bundle $L \to M$ together with a biderivation of $L$:
   \[
   \Gamma (L) \times \Gamma (L) \to \Gamma (L)
   \]
   denoted either $J$ or $\{-,-\}$, such that $[J,J]^{SJ} = 0$, i.e.~$\{-,-\}$ is a Lie bracket. Triple $(M, L, J)$ is then called a \emph{Jacobi manifold}. The homogenization $\widetilde J$ of $J$ is a Poisson structure on $\widetilde M$ with the additional property that $\mathcal L_Z \widetilde J = - \widetilde J$. Triple $(\widetilde M, \widetilde J, Z)$ is sometimes called the \emph{Poissonization} of $(M, L, J)$.
   
Interpret $J$ as a $2$-form 
\[
J :\wedge^2 \mathfrak J^1 L \to L.
\]
When it is non-degenerate, then its inverse
\[
\omega = J^{-1} : \wedge^2 DL \to L
\]
is a symplectic Atiyah $2$-form, and conversely, the inverse of a symplectic Atiyah $2$-form is a Jacobi structure. It follows that there is a one-to-one correspondence between non-degenerate Jacobi structures on $L$ and contact structures $H$ such that $TM/H = L$. Notice that $J$ is a non-degenerate Jacobi structure iff its homogenization $\widetilde J$ is a non-degenerate Poisson structure, and, in this case
\[
\widetilde J{}^{-1} = \widetilde{J^{-1}}.
\]
i.e.~\emph{symplectization and Poissonization agree}.
    \end{example}
    
    \begin{example}[Nijenhuis structures] \label{ex:Nijenhuis}
Let $U : DL \to DL$ be an Atiyah $(1,1)$-tensor on a line bundle $L \to M$. We say that $U$ is a \emph{Nijenhuis Atiyah tensor} if its Lie algebroid Nijenhuis torsion vanishes: $\mathcal T_U = 0$. It immediately follows from (\ref{eq:hom_Frol}) that $U$ is a Nijenhuis Atiyah tensor iff its homogenization $\widetilde U$ is a Nijenhuis structure on $\widetilde M$. In this case, the pair $(\widetilde U, Z)$ is a \emph{homogeneous Nijenhuis structure}, ie. $\mathcal L_Z \widetilde U = 0$.

Recall that a \emph{Poisson-Nijenhuis structure} on a manifold $M$ is a pair $(\pi, N)$ where
\begin{enumerate}
\item[$\triangleright$] $\pi$ is a Poisson structure;
\item[$\triangleright$] $N : TM \to TM$ is a Nijenhuis structure;
\item[$\triangleright$] $\pi_N := \pi (N^\ast -, -)$ is skew-symmetric;
\item[$\triangleright$] $\mathcal L_{\pi^\sharp \eta} N^\ast \theta - \mathcal L_{\pi^\sharp \theta} N^\ast \eta - d \langle \pi_N ,\alpha\wedge \beta\rangle = N^\ast [\eta, \theta]_\pi$,
\end{enumerate} 
 for all $\eta, \theta \in \Omega^1 (M)$. Here $N^\ast : T^\ast M \to T^\ast M$ is the transpose of $N$, and we denoted by
$
[-,-]_\pi
$
 the Lie algebroid bracket in the cotangent algebroid associated to the Poisson structure $\pi$, i.e.~the Lie bracket on $1$-forms given by formula
\begin{equation}\label{eq:brack_pi}
[\eta, \theta]_\pi = \mathcal L_{\pi^\sharp \eta} \theta - \mathcal L_{\pi^\sharp \theta} \eta- d \langle \pi, \eta \wedge \theta \rangle
\end{equation}
for all $\eta, \theta \in \Omega^1 (M)$, where $\pi^\sharp : T^\ast M \to TM$ is the vector bundle map induced by $\pi$. One can define a \emph{Jacobi-Nijenhuis structure} exactly in the same way, but replacing the Poisson structure with a Jacobi structure and the Nijenhuis structure with a Nijenhuis Atiyah tensor (see Section \ref{Sec:JN} for more details). So, let $(J, U)$ be a pair consisting of a bi-derivation and an Atiyah 1-1 tensor. It's clear that $(J,U)$ is a Jacobi-Nijenhuis structure iff $(\widetilde J, \widetilde U)$ is a Poisson-Nijenhuis structure, and, in this case, the triple $(\widetilde J, \widetilde U, Z)$ is  a \emph{homogeneous Poisson-Nijenhuis structure}, i.e.~$\mathcal L_Z \widetilde J = - \widetilde J$, and $\mathcal L_Z \widetilde U = 0$.

Finally, let $(\pi, N)$ be a Poisson-Nijenhuis structure. If $\pi$ is non-degenerate and $\omega = \pi^{-1}$ is the associated symplectic form, we say that $(\omega, N)$ is a \emph{symplectic-Nijenhuis} structure. This is the same as to say that 
\begin{enumerate}
\item[$\triangleright$] $\omega$ is a symplectic form,
\item[$\triangleright$] $N$ is a Nijenhuis structure, 
\item[$\triangleright$] $\omega_N := \omega (N -, -)$ is skew-symmetric, and
\item[$\triangleright$] $d \omega_N = 0$.
\end{enumerate} 
One defines \emph{contact-Nijenhuis} structures in a similar (and obvious) way (see Section \ref{Sec:JN}). The homogenization of a contact-Nijenhuis structure is a symplectic-Nijenhuis structure.
   \end{example}

    \subsection{The complex case}\label{Sec:complex_case}
    
 Let $X = (M, j)$ be a complex manifold. Here $j : TM \to TM$ is the complex structure. Consider a holomorphic line bundle $L \to X$. We will write $L \to M$ when we want to forget about the complex structures, and want to see $L$ as a real, rank $2$ vector bundle over $M$. Let $L^\ast \to X$ be the complex dual of $L$, put $\widetilde X = L^\ast \smallsetminus 0$, and denote by $\widetilde p : \widetilde X \to X$ the projection. We call $\widetilde X$ the \emph{complex homogenization} of $L \to X$. It is a holomorphic principal $\mathbb C^\times$-bundle over $X$ and every holomorphic principal $\mathbb C^\times$-bundle arises in this way. More precisely there is an equivalence between the category of holomorphic line bundles and the category of holomorphic principal $\mathbb C^\times$-bundles. Even more, restricting to holomorphic sections, holomorphic derivations, and holomorphic jets, one easily sees that the \emph{homogenization construction} described above has a precise (and obvious) analogue in the holomorphic realm. For instance, holomorphic derivations of $L \to X$ correspond to \emph{homogeneous holomorphic vector fields} on $\widetilde X$, where, by ``homogeneous'', we mean that they commute with the (holomorphic) Euler vector field. The aim of this subsection is remarking that, in this complex case, the homogenization can be actually performed in two steps. First we pass from $L \to X$ to a real $U(1)$-principal bundle $\widehat M \to M$ equipped with a canonical real line bundle $\widehat L \to \widehat M$. Second we homogenize $\widehat L \to \widehat M$ and see that this homogenization agrees with $\widetilde X$. At any step we keep track of the complex structures in a suitable way. The idea behind this construction comes from the following trivial remark. The structure group $\mathbb C^\times$ of the principal bundle $\widetilde X \to X$ factorizes as $\mathbb C^\times \simeq U(1) \times \mathbb R^\times $. Isomorphism $\mathbb C^\times \simeq  U(1) \times \mathbb R^\times $ is given by
 \[
U(1) \times \mathbb R^\times  \to \mathbb C^\times, \quad (\varphi, r) \mapsto r e^{i \varphi/2}.
 \]
 Roughly, to construct the homogenization $\widetilde X$ of $L \to X$, one can take care of the $U(1)$-factor first, and the $\mathbb R^\times$ factor later on. We now describe how to do this in a precise way. Actually we already discussed this construction in \cite[Section 3.5]{VW2016b}. We report the details here for completeness. We also add some new details.
 
%\subsubsection{Taking care of the $U(1)$-action}

 First of all, in this setting, we denote by $\widetilde M$ the real manifold underlying $\widetilde X$, and by $j_{\widetilde M}$ the complex structure on it, so $\widetilde X = (\widetilde M, j_{\widetilde M})$. Now, regard $L^\ast \to M$ as a real vector bundle and take its real projective bundle $\widehat M := \mathbb R \mathbb P (L^\ast)$. Denote by $\widehat p : \widehat M \to M$ the projection. By construction $\widehat p : \widehat M \to M$ is a principal $U(1)$-bundle with group action given by
 \[
 U(1) \times \widehat M \to \widehat M, \quad (\varphi , [\phi]) \mapsto \varphi.[\phi] := [ e^{i \varphi/2} \phi], \quad \phi \in L^\ast\smallsetminus 0.
 \]
Now, let $\widehat L \to \widehat M$ be the dual of the tautological bundle.

\begin{remark}\label{rem:U(1)_hom}
By definition, the homogenization $\widehat L{}^\ast \smallsetminus 0$ of $\widehat L$ is exactly $\widetilde M$.
\end{remark}

We use Remark \ref{rem:U(1)_hom} to interpret the pair $(\widehat M, \widehat L)$ as an intermediate step towards the complex homogenization $\widetilde X$ of $L \to X$. We now discuss the main properties of $(\widehat M, \widehat L)$. In the following we interpret $\widetilde M$ as the homogenization of $\widehat L$. There are two main structures on $\widetilde M$:
\begin{itemize}
\item[$\triangleright$] the complex structure $j_{\widetilde M} : T \widetilde M \to T \widetilde M$
\item[$\triangleright$] the holomorphic Euler vector field $Z$.
\end{itemize}
Both can be seen as homogeneous tensors. Hence they induce Atiyah tensors on $\widehat L$. To see this, begin with $Z$, and let $\zeta = 2 \operatorname{Re} Z$, so that 
\[
Z = \frac{1}{2} (\zeta - i j_{\widetilde M} \zeta).
\]
Clearly $\zeta$ is the (restriction of the) Euler vector field of the real vector bundle $L^\ast \to M$. Identities
\[
[\zeta, j_{\widetilde M}]^{FN} = \mathcal L_\zeta j_{\widetilde M} = 0, \quad \text{and} \quad [\zeta, j_{\widetilde M} \zeta] = \mathcal L_\zeta j_{\widetilde M} \zeta = 0
\]
then show that
\begin{itemize}
\item[$\triangleright$] $j_{\widetilde M}$ is a homogeneous $(1,1)$-tensor,
\item[$\triangleright$] $j_{\widetilde M} \zeta$ is a homogeneous vector field.
\end{itemize}
Accordingly, 
\begin{itemize}
\item[$\triangleright$] $j_{\widetilde M}$ is the homogenization of an Atiyah $(1,1)$-tensor $\widehat j$ on $\widehat L$,
\item[$\triangleright$] $j_{\widetilde M} \zeta$ is the homogenization of derivation $\widehat j \mathbb 1$ of $\widehat L$.
\end{itemize}
Additionally, it follows from the fact that $j_{\widetilde M}$ is a complex structure, that $\widehat j$ is a complex structure on the gauge algebroid of $\widehat L$, i.e.
\[
\widehat j{}^2 = - \mathbb 1, \quad \text{and} \quad \mathcal T_{\widehat j} = 0.
\]

We conclude this section discussing the relationship between the line bundle $\widehat L$, its gauge algebroid and its first jet bundle, with the holomorphic line bundle $L$, its \emph{holomorphic gauge algebroid}, and its \emph{holomorphic first jet bundle} respectively. 

\begin{proposition}\label{prop:L^hat}
There is a one-to-one correspondence, denoted $\lambda \mapsto \widehat \lambda$, between smooth sections $\lambda$ of $L$ and sections $\widehat \lambda$ of the complexification $\widehat L \otimes \mathbb C \to \widehat M$ such that $(\widehat j \mathbb 1 - i) \widehat \lambda = 0$. Geometrically, there is a canonical regular complex line bundle map
\[
\begin{array}{c}
\xymatrix{ \widehat L \otimes \mathbb C \ar[r]^-{\widehat p_L} \ar[d] & L \ar[d] \\
\widehat M \ar[r]^-{\widehat p} & M
}
\end{array}.
 \]
In particular complex line bundle $\widehat L \otimes \mathbb C$ is canonically isomorphic to the pull-back $\widehat p{}^\ast L$. The pull-back of sections identifies $\lambda \in \Gamma (L)$ with $\widehat \lambda \in \Gamma (\widehat L \otimes \mathbb C)$. Finally, $\lambda$ is holomorphic iff $(\widehat j{}^\dag - i) \mathfrak j^1 \widehat \lambda = 0$.
\end{proposition}

The last equality in the statement deserves some explanations. Recall that $\widehat j$ is, in particular, an Atiyah $(1,1)$-tensor on $\widehat L$ and regard it as a vector bundle endomorphism $\widehat j : D\widehat L \to D \widehat L$. We denote by $\widehat j{}^\dag : \mathfrak J^1 \widehat L \to \mathfrak J^1 \widehat L$ the transpose of $\widehat j$ twisted by $L$, i.e.~its adjoint with respect to the duality pairing $\mathfrak J^1 \widehat L \otimes D \widehat L \to \widehat L$ (and similarly for all other Atiyah $(1,1)$-tensors).

\begin{proof}
Let $\lambda$ be a section of $L \to X$. First of all, notice that, if $\lambda$ is holomorphic, then it corresponds to a homogeneous holomorphic function $\widetilde \lambda$ on $\widetilde X$. If $\lambda$ is not holomorphic, one can still define a smooth complex function $\widetilde \lambda$ on $\widetilde X$ by the usual formula
\[
\widetilde \lambda (\epsilon) = \langle \epsilon, \lambda_{\widetilde p(\epsilon)} \rangle, \quad \epsilon \in \widetilde X.
\]
Functions $ f : \widetilde X \to \mathbb C$ of this form can be characterized as those smooth functions such that
\[
\mathcal L_Z f = f, \quad \text{and} \quad \mathcal L_{\overline Z} f = 0,
\]
or, equivalently,
\[
\mathcal L_\zeta f = f, \quad \text{and} \quad \mathcal L_{j_{\widetilde M}\zeta} f = i f.
\]

A (non-necessarily holomorphic) section $\lambda$ of $L$ does also define a section $\widehat \lambda$ of $\widehat L \otimes \mathbb C \to \widehat M$ as follows. Let $\phi \in L^\ast \smallsetminus 0$. We see $[\phi] \in \widehat M$ as the real line in $L^\ast_{\widehat p [\phi]}$ spanned by $\phi$. A point in $\widehat L \otimes \mathbb C$ over $[\phi]$ can then be seen as an $\mathbb R$-linear map $\mathsf{f} : [\phi] \to \mathbb C$. We define $\widehat \lambda_{[\phi]}$ by
\begin{equation}\label{eq:lambda^hat}
\widehat \lambda_{[\phi]} : [\phi] \to \mathbb C, \quad \phi \mapsto \langle \phi , \lambda_{\widehat p [\phi]} \rangle.
\end{equation}
It immediately follows from (\ref{eq:lambda^hat}) that the homogenization of $\widehat \lambda$ agrees with $\widetilde \lambda$. Hence, sections of $\widehat L \otimes \mathbb C$ of the form $\widehat \lambda$ can be characterized as those sections such that $(\widehat j \mathbb 1) \widehat \lambda = i \widehat \lambda$, as claimed.

Finally, we define $\widehat p_L $. Let $\phi \in L^\ast \smallsetminus 0$, and let $\mathsf f : [\phi] \to \mathbb C$ be a point in $\widehat L$ over $[\phi] \in \widehat M$. Then $\widehat p_L (\mathsf f)$ is the point in $L_{\widehat p [\phi]}$ implicitly defined by
\[
\langle \phi, \widehat p_L (\mathsf f) \rangle = \mathsf f (\phi).
\]
Clearly, $\widehat p_L (\mathsf f)$ is well-defined, i.e.~it is independent of the choice of $\phi$, and it is a regular complex vector bundle map. From (\ref{eq:lambda^hat}), pull-back section $\widehat p^\ast \lambda$ identifies with $\widehat \lambda$. The last part of the statement follows from the fact that $\lambda$ is holomorphic iff $\widetilde \lambda$ is so iff $(\widetilde j{}^\ast - i) d\widetilde \lambda = 0$.
 \end{proof}
 
 \begin{remark}
 Proposition \ref{prop:L^hat} shows that, while $L$ has no underlying real structure, i.e.~it is not the complexification of a real line bundle, in general, its pull-back $\widehat p{}^\ast L$ always is.
 \end{remark}

Notice that Proposition \ref{prop:L^hat} does not involve the complex structure on $M$ nor the holomorphic vector bundle structure on $L$, but only its complex vector bundle structure, except for what concerns holomorphic sections of $L$. We now discuss how does the holomorphic vector bundle structure on $L \to X$ reflects on $\widehat L$. To do this we consider the gauge algebroid of the real vector bundle $L \to M$. We denote it by $D_{\mathbb R} L$ in order to distinguish it from the \emph{holomorphic gauge algebroid} of the holomorphic line bundle $L \to X$ \cite{VW2016b}. The holomorphic vector bundle structure on $L \to X$ is equivalent to a fiber-wise complex structure $j_{\mathrm{fw}} : L \to L$, and an Atiyah $(1,1)$-tensor
\[
j_{DL} : D_{\mathbb R} L \to D_{\mathbb R} L,
 \]
such that \cite[Section 2.3]{VW2016}
\begin{itemize}
\item[$\triangleright$] $j_{DL}$ is \emph{almost complex}, i.e.~$(j_{DL})^2 = -\mathbb 1$,
\item[$\triangleright$] $j_{DL}$ is \emph{integrable}, i.e.~$\mathcal T_{j_{DL}} = 0$,
\item[$\triangleright$] the symbol $\sigma : D_{\mathbb R} L \to TM$ intertwines $j_{DL}$ and $j$,
\item[$\triangleright$] the restriction of $j_{DL}$ to endomorphisms $\operatorname{End}_{\mathbb R}\! L$ agrees with the map $h \mapsto j_{\mathrm{fw}} \circ h$.
\end{itemize}

Now, the holomorphic gauge algebroid $DL \to X$ of $L \to X$ is the complex subalgebroid of $D_{\mathbb R} L$ consisting of $\mathbb C$-linear derivations. Not only $DL \to X$ is a complex Lie algebroid, it is actually a holomorphic Lie algebroid. There is an alternative manifestation of the holomorphic gauge algebroid, denoted $D^{1,0} L$, which is equivalent to $DL$ up to a canonical isomorphism of holomorphic Lie algebroids. In order to define $D^{1,0} L$, consider first the complex (beware not \emph{holomorphic}) Lie algebroid $D_{\mathbb C} L \to M$ whose sections are complex derivations of $L \to M$, i.e.~$\mathbb C$-linear maps $\Delta : \Gamma (L) \to \Gamma (L)$ satisfying a Leibniz rule
\[
\Delta (f \lambda) = \sigma (\Delta) f \cdot \lambda + f \Delta \lambda, \quad f \in C^\infty (M, \mathbb C), \quad \lambda \in \Gamma (L),
\]
for some complex vector field $\sigma (\Delta) \in \mathfrak X (M) \otimes \mathbb C$ (the \emph{symbol} of $\Delta)$. We denote by $D^{1,0} L$ the complex subbundle in $D_{\mathbb C} L$ consisting of derivations whose symbol is in $\mathfrak X^{1,0} (X) = \Gamma (T^{1,0} X)$. There is an isomorphism $DL \simeq D^{1,0} L$ given by
\[
D^{1,0} L \to DL, \quad \Delta \mapsto \Delta + \overline{\partial}_{\overline{\sigma (\Delta)}}.
\]
We refer to \cite{VW2016b} for more details.  

Denote by $D_{\mathbb R} \widehat L$ the gauge algebroid of the real line bundle $\widehat L \to \widehat M$. 

\begin{proposition} \label{prop:hom_der_compl}
There is a one-to-one correspondence, denoted $\Delta \mapsto \widehat \Delta$, between smooth sections $\Delta$ of $D^{1,0}L$ and sections $\widehat \Delta$ of $D_{\mathbb R} \widehat L$ such that $[\widehat j \mathbb 1, \widehat \Delta] = 0$. Geometrically, there is a canonical regular complex vector bundle map
\[
\begin{array}{c}
\xymatrix{ D_{\mathbb R} \widehat L \ar[r]^-{\widehat p_D} \ar[d] & D^{1,0} L \ar[d] \\
\widehat M \ar[r]^-{\widehat p} & M
}
\end{array}.
\]
In particular $D_{\mathbb R} \widehat L $ is canonically isomorphic to the pull-back $\widehat p{}^\ast D^{1,0}L$. The pull-back of sections identifies section $\Delta$ of $D^{1,0}L$ with $\widehat \Delta$. Finally, derivation $\Delta$ is holomorphic iff $[ \widehat j, \widehat \Delta]^{FN}_D = 0$.
\end{proposition}

\begin{proof}
Let $\Delta$ be a section of $D^{1,0} L \to X$. If $\Delta$ is holomorphic, then it corresponds to a homogeneous holomorphic vector field $\widetilde \Delta \in \mathfrak X^{1,0} (\widetilde X)$ on $\widetilde X$. If $\Delta$ is not holomorphic, one  uniquely defines a smooth section of $T^{1,0} \widetilde X$ by the same formula
\[
\widetilde \Delta (\widetilde \lambda) = \widetilde{\Delta \lambda}, \quad \lambda \in \Gamma (L).
\]
Sections of $T^{1,0} \widetilde X$ of this form can be characterized as those sections
\[
U = \frac{1}{2} (\upsilon - i j_{\widetilde M} \upsilon), \quad \upsilon \in \mathfrak X (\widetilde M),
\]
such that
\[
\mathcal L_Z U = \mathcal L_{\overline Z} U = 0,
\]
 or, equivalently,
\[
\mathcal L_\zeta  U = \mathcal L_{j_{\widetilde M} \zeta} U = 0.
\]
 or, yet in other terms,
\[
\mathcal L_\zeta  \upsilon = \mathcal L_{j_{\widetilde M} \zeta} \upsilon = 0.
\]

So $\upsilon$ is the homogenization of a derivation $\widehat \Delta$ of $\widehat L$ such that $[\widehat j \mathbb 1, \widehat \Delta] = 0$. Notice that $\widehat \Delta$ is uniquely determined by the condition that
\begin{equation}\label{eq:Delta^hat}
\widehat \Delta{}^{1,0} \widehat \lambda = \widehat{\Delta \lambda}, \quad \text{where} \quad \widehat \Delta{}^{1,0} = \frac{1}{2} \left(\widehat \Delta -i \widehat j \widehat \Delta  \right)
\end{equation}
for all $\lambda \in \Gamma (L)$.

Now, we define $\widehat p_D$. It is clear how to do this. We map a derivation $\Delta \in D_{\mathbb R} \widehat L$ over a point $[\phi] \in \widehat M$, to derivation $\widehat p_D (\Delta) \in D_{\widehat p [\phi]}^{1,0} L$ implicitly defined by
\begin{equation}\label{eq:p_D^hat}
\langle \phi, \widehat p_D (\Delta) \lambda \rangle = (\Delta^{1,0} \widehat \lambda ) (\phi) \quad \text{where} \quad \Delta^{1,0} = \frac{1}{2} \left( \Delta - i \widehat j \Delta \right),
\end{equation}
for all $\lambda \in \Gamma (L)$. Here we interpret the point $\Delta^{1,0} \widehat \lambda \in \widehat L_{[\phi]} \otimes \mathbb C$ as a linear map $[\phi] \to \mathbb C$.

Definitions (\ref{eq:Delta^hat}) and (\ref{eq:p_D^hat}) immediately imply that pull-back section $\widehat p_D^\ast \Delta$ identifies with $\widehat \Delta$. The last part of the statement follows from the following chain of equivalences 

\[
\begin{aligned}
& \text{a section $\Delta$ of $D^{1,0}L$ is holomorphic} \\
& {} \Leftrightarrow \text{$\widetilde \Delta$ is a holomorphic vector field} \\
& {} \Leftrightarrow\mathcal L_{\widetilde \Delta} j_{\widetilde M}= [\widetilde \Delta, j_{\widetilde M}]^{FN} = 0 \\
& {} \Leftrightarrow[ j_{\widetilde M}, 2 \operatorname{Re} \widetilde \Delta]^{FN} = 0.
\end{aligned}
\]

\end{proof}

Denote by $\mathfrak J^1_{\mathbb R} \widehat L$ the first jet bundle of $\widehat L \to \widehat M$, and by $\mathfrak J^1 L$ the holomorphic first jet bundle of $L$. Recall that $\mathfrak J^1_{\mathbb R} \widehat L = \operatorname{Hom} (D_{\mathbb R} \widehat L, \widehat L)$. Similarly, $\mathfrak J^1 L = \operatorname{Hom}_{\mathbb C} (D^{1,0} L, L)$. So, a simple \emph{duality argument} allows one to prove the following

\begin{proposition}\label{prop:hom_jet_compl}
There is a one-to-one correspondence, denoted $\psi \mapsto \widehat \psi$, between smooth sections $\psi$ of $\mathfrak J^1 L$ and section $\widehat \psi$ of $\mathfrak J^1_{\mathbb R} \widehat L$ such that $(\mathcal L_{\widehat j \mathbb 1} - \widehat j{}^\dag)\widehat \psi = 0$. Geometrically, there is a canonical regular complex vector bundle map
\[
\begin{array}{c}
\xymatrix{ \mathfrak J^1_{\mathbb R} \widehat L \ar[r]^-{\widehat p_{\mathfrak J^1}} \ar[d] & \mathfrak J^1 L \ar[d] \\
\widehat M \ar[r]^-{\widehat p} & M
}
\end{array}.
\]
In particular $ \mathfrak J^1_{\mathbb R} \widehat L$ is canonically isomorphic to the pull-back $\widehat p{}^\ast \mathfrak J^1 L$. The pull-back of sections identifies section $\psi$ of $\mathfrak J^1 L$ with $\widehat \psi$. Finally section $\psi$ of $\mathfrak J^1 L$ is holomorphic iff 
\[
\mathcal L_{\widehat j} \widehat \psi - d_D \widehat j{}^\dag \widehat \psi = d_D \widehat \psi + \mathcal L_{\widehat j} \widehat j{}^\dag \widehat \psi = 0
\]
(where $\mathcal L_{\widehat j}$ is the Lie derivative of Atiyah forms along the Atiyah $(1,1)$-tensor $\widehat j$).
\end{proposition}

\begin{proof}
We only sketch the proof leaving the details to the reader. Let $\psi$ be a (non-necessarily holomorphic) section of $\mathfrak J^1 L $. There is a unique $1$-form $\Omega^{1,0} (\widetilde X)$ such that 
\[
\langle \widetilde \psi, \widetilde \Delta \rangle = \widetilde{\langle \Delta, \psi \rangle}, 
\]
for all sections $\Delta$ of $D^{1,0} L$. Forms in  $\Omega^{1,0} (\widetilde X)$ of this kind can be characterized as those forms
\[
\psi = \frac{1}{2} (\varphi - i j_{\widetilde M}^\ast \varphi), \quad \varphi \in \Omega^1 (\widetilde M),
\]
such that
\[
\mathcal L_\zeta  \varphi - \varphi = \mathcal L_{j_{\widetilde M} \zeta} \varphi - j_{\widetilde M}^\ast \varphi = 0.
\]
So $\varphi$ is the homogenization of an Atiyah $1$-form $\widehat \psi$ on $\widehat L$ such that $(\mathcal L_{\widehat j \mathbb 1} - \widehat j{}^\dag)  \widehat \psi = 0$. Actually, $\widehat \psi$ is uniquely determined by the condition that
\begin{equation}
\langle \widehat \psi{}^{1,0}, \widehat \Delta{}^{1,0} \rangle = \widehat{\langle \psi, \Delta \rangle}, \quad \text{where} \quad \widehat \psi{}^{1,0} = \frac{1}{2} \left(\widehat \psi -i \widehat j{}^\dag \widehat \psi  \right)
\end{equation}
for all sections $\Delta$ of $DL$.

It is clear how to define $\widehat p_{\mathfrak J^1}$. Any point $\psi$ in $\mathfrak J^1_{\mathbb R} \widehat L$ is of the form
\[
2 \operatorname{Re} \mathfrak j_{[\phi]}^1 \widehat \lambda
\]
for some $\lambda \in \Gamma (L)$, and $[\phi] \in \widehat M$. We put
\[
\widehat p_{\mathfrak J^1} (\psi) = \mathfrak j^{1,0}_{\widehat p[\phi]} \lambda
\]
(see \cite{VW2016b} for the precise definition of $\mathfrak j^{1,0}$). Equivalently
\[
\langle \widehat p_{\mathfrak J^1} (\psi), \Delta \rangle = \langle \psi, \widehat p_D|_{D_{[\phi]} \widehat L}^{-1} \Delta \rangle,
\]
for all $\Delta \in D^{1,0} L$. 

The last part of the statement follows from the fact that a real form $\omega$ on a complex manifold $(M, j)$ is the real part of a holomorphic form iff
\[
\mathcal L_j \omega - d j^\ast \omega = d \omega + \mathcal L_j j^\ast \omega = 0.
\]

\end{proof}

We summarize the above discussion with the following commutative diagrams

\[
\xymatrix@C=15pt@R=15pt{ & & & \mathbb C_{\widetilde X} \ar[dd] \ar[dlll]_-{\mathbb R^\times} \ar[ddll]^-{\mathbb C^\times} \\
 \widehat L \otimes \mathbb C \ar[dd] \ar[dr]^-{U(1)} & & & \\
 & L \ar[dd] & & \widetilde X \ar[dlll]|!{[ll];[lldd]}\hole \ar[ddll]  \\
 \widehat M \ar[dr] & & & \\
 & X & &
} \quad 
\xymatrix@C=15pt@R=15pt{ & & & T^{1,0}\widetilde X \ar[dd] \ar[dlll]_-{\mathbb R^\times} \ar[ddll]^-{\mathbb C^\times} \\
 D_{\mathbb R} \widehat L \ar[dd] \ar[dr]^-{U(1)} & & & \\
 & D^{1,0} L \ar[dd] & & \widetilde X \ar[dlll]|!{[ll];[lldd]}\hole \ar[ddll]  \\
 \widehat M \ar[dr] & & & \\
 & X & &
} \quad 
\xymatrix@C=15pt@R=15pt{ & & & T^\ast{}^{1,0} \widetilde X \ar[dd] \ar[dlll]_-{\mathbb R^\times} \ar[ddll]^-{\mathbb C^\times} \\
 \mathfrak J^{1}_{\mathbb R} \widehat L \ar[dd] \ar[dr]^-{U(1)} & & & \\
 & \mathfrak J^1 L \ar[dd] & & \widetilde X \ar[dlll]|!{[ll];[lldd]}\hole \ar[ddll]  \\
 \widehat M \ar[dr] & & & \\
 & X & &
}
\]

The top horizontal arrows are principal bundles in the category of vector bundles with regular vector bundle maps as morphisms.

\begin{remark}
It follows from Propositions \ref{prop:hom_der_compl} and \ref{prop:hom_jet_compl}, that, for all $l,m$ there is a one-to-one correspondence between smooth/holomorphic sections $T$ of the \emph{bundle of holomorphic Atiyah tensors}
\[
(D^{1,0}L)^{\ast \otimes_{\mathbb C} l} \otimes_{\mathbb C} (\mathfrak J^1 L)^{\ast \otimes_{\mathbb C} m} \otimes_{\mathbb C} L
\]
and sections $\widehat T$ of
\[
(D_{\mathbb R} \widehat L)^{\ast \otimes l} \otimes (\mathfrak J_{\mathbb R}^1 \widehat L)^{\ast \otimes m} \otimes \widehat L \otimes \mathbb C
\]
satisfying certain identities involving $\widehat j$. For instance, in the case when $T = J$ is a holomorphic, skew-symmetric biderivation such that $[J,J]^{SJ} = 0$, these identities say that $(\widehat J, \widehat j)$ is a Jacobi-Nijenhuis structure \cite{VW2016b}.

%Geometrically, there is a regular vector bundle map
%\[
%\xymatrix{(D_{\mathbb R} \widehat L)^{\ast \otimes l} \otimes (\mathfrak J_{\mathbb R}^1 \widehat L)^{\ast \otimes m} \otimes \widehat L \otimes \mathbb C \ar[r] \ar[d] & (D^{1,0}L)^{\ast \otimes_{\mathbb C} l} \otimes_{\mathbb C} (\mathfrak J^1 L)^{\ast \otimes_{\mathbb C} m} \otimes_{\mathbb C} L \ar[d] \\
%\widehat M \ar[r] & M
%}
%\]
\end{remark}
     
%\section{Homogenization of line bundles in the categories of Lie groupoids}
\subsection{Homogenization of line bundle groupoids}
In Section \ref{Sec:Cont_Group} we will deal with contact groupoids, i.e.~Lie groupoids $\mathcal G \rightrightarrows M$ equipped with a compatible contact structure $H \subset T \mathcal G$. In this situation, the normal bundle $L = T\mathcal G / H$ is equipped with a groupoid structure so that $L \to \mathcal G$ is a \emph{VB-groupoid}, i.e.~a \emph{vector bundle in the category of Lie groupoids}. In this section we discuss homogenization in this setting, and show that the homogenization of $L$ is a Lie groupoid itself, actually a \emph{principal $\mathbb R^\times$-bundle in the category of Lie groupoids} (see also \cite{BGG2015}). We begin recalling the definition of a VB-groupoid. For more details we refer to \cite{M1992, M2000}, see also \cite{GM2015, BCdH2016}.
	
\begin{definition}\label{def:VB_group}
A \emph{VB-groupoid} $(\Omega \rightrightarrows E; \mathcal G \rightrightarrows M)$ is a vector bundle in the category of Lie groupoids, i.e.~a diagram
\[
\xymatrix{ \Omega \ar[d] \ar@<+2pt>[r] \ar@<-2pt>[r]& E \ar[d] \\
\mathcal G \ar@<+2pt>[r] \ar@<-2pt>[r] & M
}
\]
denoted shortly  by $(\Omega, E; \mathcal G, M)$, such that
\begin{itemize}
\item[(VB1)] rows are Lie groupoids,
\item[(VB2)] columns are vector bundles,
\item[(VB3)] all vector bundle structure maps are Lie groupoid maps.
\end{itemize}
%Given a VB-groupoid, we will usually denote by $s,t,m,1, \iota$ (resp.~$\bar s, \bar t, \bar m, \bar 1, \bar{\iota}$) the source, the target, the multiplication, the unit and the inversion in the \emph{bottom groupoid} $\mathcal G \rightrightarrows M$ (resp.~\emph{top groupoid} $\Omega \rightrightarrows E$). Multiplication and inversion will be also denoted by a dot $\cdot$ and a superfix ${}^{-1}$ respectively. %The spaces of composable arrows in $\mathcal G, \Omega$ will be denoted by $\mathcal G^{(2)}, \Omega^{(2)}$.
The \emph{core} of $(\Omega, E; \mathcal G, M)$ is vector bundle $C := 1^\ast (\ker \bar s )$.\end{definition}

\begin{remark}
Condition (VB3) in Definition \ref{def:VB_group} can be partially relaxed, requiring only that multiplication by $r $ in the fibers of $\Omega \to \mathcal G$ is a Lie groupoid map, for all $r \in \mathbb R$. This results often in great simplifications.
\end{remark}

\begin{remark}\label{rem:Euler_mult}
The Euler vector field of a VB-groupoid is multiplicative.
\end{remark}

\begin{example}[Tangent VB-groupoid]
Let $\mathcal G \rightrightarrows M$ be a Lie groupoid. Then $(T \mathcal G \rightrightarrows TM; \mathcal G \rightrightarrows M)$ is a VB-groupoid called the \emph{tangent VB-groupoid}. The structure maps of the top groupoid $T \mathcal G \rightrightarrows TM$ are the tangent maps to the structure maps of the bottom groupoid $\mathcal G \rightrightarrows M$. The core of $(T \mathcal G \rightrightarrows TM; \mathcal G \rightrightarrows M)$ is canonically isomorphic to the Lie algebroid $A$ of $\mathcal G$.
\end{example}

\begin{example}
Let $(\Omega \rightrightarrows E; \mathcal G \rightrightarrows M)$ be a VB-groupoid, with core $C$. Then $(\Omega^\ast \rightrightarrows C^\ast; \mathcal G \rightrightarrows M)$ is a VB-groupoid with core $E^\ast$, called the \emph{dual VB-groupoid}. For more details about the the dual VB-groupoid see \cite{M2005}. In particular, given a Lie groupoid $\mathcal G \rightrightarrows M$ with Lie algebroid $A$, $(T^\ast \mathcal G \rightrightarrows A^\ast, \mathcal G \rightrightarrows M)$ is a VB-groupoid with core $T^\ast M$, called the \emph{cotangent VB-groupoid}.
\end{example}

\begin{example}
Let $\mathcal G \rightrightarrows M$ be a Lie groupoid acting on a vector bundle $q : E \to M$. Denote by $s,t,1$ source, target, and unit of $\mathcal G$. Then $(s^\ast E \rightrightarrows E; \mathcal G  \rightrightarrows M)$ is a VB-groupoid with trivial core. The structure maps of $s^\ast E \rightrightarrows E$ are the following:
\begin{itemize}
\item[$\triangleright$] the source $\tilde s$ and target $\tilde t$ are $\tilde s (g,e) = e$ and $\tilde t (g, e) = g.e$; 
\item[$\triangleright$] the multiplication $\tilde m$ is $\tilde m ((g,e),(g',e')) = (gg',e')$;
\item[$\triangleright$] the unit $\tilde 1$ and the inversion are $\tilde 1 (e) = (1_{q(e)}, e)$ and $(g,e)^{-1} = (g^{-1}, g.e)$;
\end{itemize}
where $g,g' \in \mathcal G$, and $e,e' \in E$. All VB-groupoids with trivial core are of this kind. 
\end{example}

\begin{example}
Let $\mathcal G \rightrightarrows M$ be a Lie groupoid acting on a vector bundle $C$. Then $(t^\ast C  \rightrightarrows 0_M; \mathcal G  \rightrightarrows M)$ is a VB-groupoid with core $C$. Here $0_M \to M$ is the trivial vector bundle. The structure maps in $t^\ast C \rightrightarrows 0_M$ are the following
\begin{itemize}
\item[$\triangleright$] the source $\tilde s$ and target $\tilde t$ are $\tilde s (g,c) = 0_{s(g)}$ and $\tilde t (g, c) = 0_{t(g)}$; 
\item[$\triangleright$] the multiplication $\tilde m$ is $\tilde m ((g,c),(g',c')) = (gg',c+ g.c')$;
\item[$\triangleright$] the unit $\tilde 1$ and the inversion are $\tilde 1 (0_x) = (1_x, 0_x)$ and $(g,c)^{-1} = (g^{-1}, -g^{-1}.c)$;
\end{itemize}
where $g,g' \in \mathcal G$, $c,c' \in C$, and $x \in M$. All VB-groupoids $(\Omega  \rightrightarrows E; \mathcal G  \rightrightarrows M)$ with trivial \emph{side bundle} $E$ are of this kind. Additionally, VB-groupoids with trivial side bundle and VB-groupoids with trivial core are in duality. More precisely, let $(s^\ast E  \rightrightarrows E; \mathcal G  \rightrightarrows M)$ be a VB-groupoid with trivial core corresponding to an action of $\mathcal G$ on vector bundle $E$. Then the dual VB-groupoid is (canonically isomorphic to) the VB-groupoid $(t^\ast E^\ast  \rightrightarrows 0_M; \mathcal G  \rightrightarrows M)$ with trivial side bundle associated to the action of $\mathcal G$ on $E^\ast$, and vice-versa.
\end{example}

\begin{remark}\label{rem:LB-groupoids}
Let $(\mathcal L , E ; \mathcal G, M)$ be a VB-groupoid such that $\mathcal L \to \mathcal G$ is a line bundle. As the source and target maps in a Lie groupoid are always submersions by definition, it follows that either $E = 0_M$, $C = L$ is a line bundle, and $\mathcal L = t^\ast L$, or $E = L$ is a line bundle, $C = 0$, and $\mathcal L = s^\ast L$. In the former case we speak about an \emph{LB-groupoid}. In the following we discuss homogenization of LB-groupoids.
\end{remark}

Let $\mathcal G \rightrightarrows M$ be a Lie groupoid acting on a line bundle $L \to M$, and let $(t^\ast L  \rightrightarrows 0_M; \mathcal G  \rightrightarrows M)$ be the corresponding LB-groupoid. Take the dual VB-groupoid $(s^\ast L^\ast  \rightrightarrows L^\ast; \mathcal G  \rightrightarrows M)$. In a VB-groupoid all structure maps of the top groupoid are vector bundle maps, in particular they map the zero section to the zero section. Hence, if we remove the zero section from $s^\ast L^\ast$ we get a (new) Lie groupoid over $L^\ast \smallsetminus 0$. We put $\widetilde{\mathcal G} := s^\ast L^\ast \smallsetminus 0= s^\ast (L \smallsetminus 0)$ and $\widetilde M = L^\ast \smallsetminus 0$. So the resulting Lie groupoid is $\widetilde{\mathcal G} = s^\ast \widetilde M \rightrightarrows \widetilde M$. Additionally, it is clear that diagram
\[
\xymatrix{\widetilde{\mathcal G} \ar[d] \ar@<+2pt>[r] \ar@<-2pt>[r]& \widetilde M \ar[d] \\
\mathcal G \ar@<+2pt>[r] \ar@<-2pt>[r] & M
}
\]
is a principal $\mathbb R^\times$-bundle in the category of Lie groupoids, i.e.
\begin{itemize}
\item[$\triangleright$] rows are Lie groupoids,
\item[$\triangleright$] columns are principal $\mathbb R^\times$-bundles,
\item[$\triangleright$] $\mathbb R^\times$ acts on $\widetilde{\mathcal G}$ by Lie groupoid automorphisms.
\end{itemize}
Actually, all principal $\mathbb R^\times$-bundles in the category of Lie groupoids can be obtained in this way. Notice that, by linearity, the action of $\mathcal G$ on $L^\ast$ restricts to an action on $\widetilde M$ by principal bundle automorphisms, and $\widetilde{\mathcal G} \rightrightarrows \widetilde M$ is the corresponding action groupoid.

	\subsection{Multiplicative Atiyah tensors}

Multiplicative structures on Lie groupoids play an important role in Poisson geometry. Recently, Bursztyn, and Drummond, elaborating on earlier works on multiplicative forms, multivectors, and $(1,1)$-tensors, proposed a general scheme to deal with generic multiplicative tensors \cite{BD2017}. Their idea consists in viewing a tensor $\mathcal T$ on a groupoid as a real valued function $\mu_{\mathcal T}$ on a certain fibered product which is a groupoid itself (in \cite{LSX2008} the same idea has been used to discuss multiplicative bivectors, see also \cite{ISX2012}). So, it makes sense to declare that $\mathcal T$ is \emph{multiplicative} if $f_{\mathcal T}$ is a \emph{multiplicative function} i.e.~it is a Lie groupoid cocycle. In this section, we extend this idea to Atiyah tensors and we relate the multiplicativity of an Atiyah tensor to the multiplicativity of its homogenization. Our main examples will come from multiplicative contact structures, seen as symplectic Atiyah forms (see Example \ref{ex:contact}).

We begin recalling Bursztyn and Drummond idea. Let $\mathcal G \rightrightarrows M$ be a Lie groupoid with Lie algebroid $A$, and let $\mathcal T$ be an $(l,m)$-tensor on it, i.e.~a vector bundle map 
\[
(T \mathcal G )^{\otimes l} \otimes (T^\ast \mathcal G)^{\otimes m} \to \mathbb R_{\mathcal G}.
\]
Clearly, $\mathcal T$ can be regarded as a smooth function on the fibered product
\[
\mathfrak T^{l,m} \mathcal G  := \underset{\text{$l$ times}}{\underbrace{T\mathcal G  \times_{\mathcal G } \cdots \times_{\mathcal G }  T\mathcal G }} \times_{\mathcal G } \underset{\text{$m$ times}}{\underbrace{T^\ast \mathcal G  \times_{\mathcal G } \cdots \times_{\mathcal G } T^\ast \mathcal G }},
\]
i.e.~$f_{\mathcal T} : \mathfrak T^{l,m} \mathcal G \to \mathbb R$.

It is easy to see that $\mathfrak T^{l,m} \mathcal G$ is a Lie groupoid itself. The manifold of objects is the fibered product
\[
(\mathfrak T^{l,m} \mathcal G)_0 :=  \underset{\text{$l$ times}}{\underbrace{TM  \times_{M } \cdots \times_{M}  TM }} \times_{M } \underset{\text{$m$ times}}{\underbrace{A^\ast  \times_{M } \cdots \times_{M} A^\ast}},
\]
and the structure maps are induced, from those of the tangent and the cotangent groupoids, in the obvious, component-wise, way. Additionally, diagram
\[
\xymatrix{\mathfrak T^{l,m} \mathcal G \ar[d] \ar@<+2pt>[r] \ar@<-2pt>[r]& (\mathfrak T^{l,m} \mathcal G)_0 \ar[d] \\
\mathcal G \ar@<+2pt>[r] \ar@<-2pt>[r] & M
}
\]
is a \emph{fibered groupoid} (beware, \emph{not} a VB-groupoid), meaning that the fiber bundle projection $\mathfrak T^{l,m} \mathcal G \to \mathcal G$ is a groupoid map. We say that $\mathcal T$ is \emph{multiplicative} if $f_{\mathcal T}$ is a multiplicative function, i.e.
\[
f_{\mathcal T} (\Theta_1 \cdot \Theta_2) = f_{\mathcal T} (\Theta_1) + f_{\mathcal T} (\Theta_2),
\]
for all composable arrows $\Theta_1, \Theta_2 \in \mathfrak T^{l,m} \mathcal G$. When $\mathcal T$ is a differential form, a multivector or a $(1,1)$-tensor, one recovers earlier definitions. 

\begin{remark}\label{rem:mult_tens}
One can equivalently regard $f_{\mathcal T}$ as a fiber bundle map $\mathfrak T^{l,m} \mathcal G \to \mathbb R_{\mathcal G}$, also denoted by $f_{\mathcal T}$. It is immediate to check that $\mathcal T$ is multiplicative iff $f_{\mathcal T} : \mathfrak T^{l,m} \mathcal G \to \mathbb R_{\mathcal G}$ is a (necessarily fibered) groupoid map. 
We will always take this point of view without further comments.

Alternatively, one can see an $(l,m)$-tensor $\mathcal T$ on $\mathcal G$ as a map $\mathcal T': \mathfrak T^{l-1, m} \mathcal G \to T^\ast \mathcal G$ (resp., a map $\mathcal T'': \mathfrak T^{l, m-1} \mathcal G \to T\mathcal G$) in $l$ (resp.,~$m$) different ways. For instance, a $(1,1)$-tensor $\mathcal N$ can be seen as a map $\mathcal N : T \mathcal G \to T \mathcal G$. Similarly, a $2$-form $\Omega$, can be seen as a map $\Omega_\flat : T\mathcal G \to T^\ast \mathcal G$. A straightforward computations, shows that $\mathcal T$ is multiplicative iff any of the $\mathcal T'$ (resp.~$\mathcal T''$) is a groupoid map. For instance a $(1,1)$-tensor $\mathcal N$ is multiplicative iff, seen as a map $\mathcal N : T\mathcal G \to T \mathcal G$, it is a VB-groupoid map, and a $2$-form $\Omega$ is multiplicative iff the associated map $\Omega_\flat : T\mathcal G \to T^\ast \mathcal G$ is a VB-groupoid map, and this is often useful in practice.
\end{remark}

We refer to \cite{K2016} for a recent review on multiplicative structures on Lie groupoids. For later use, we only recall an alternative, equivalent definition of \emph{multiplicative differential form} which is often useful in practice. A differential form $\omega$ on a groupoid $\mathcal G$ is multiplicative if
\begin{equation}\label{eq:mult_form}
m^\ast \omega = \operatorname{pr}_1^\ast \omega + \operatorname{pr}_2^\ast \omega,
\end{equation}
where $m, \operatorname{pr}_i : \mathcal G^{(2)} \to \mathcal G$ are the multiplication and the two projections respectively, $i = 1,2$, on the manifold $\mathcal G^{(2)}$ of composable arrows.

We now pass to Atiyah tensors. The natural setting for multiplicative Atiyah tensors is that of LB-groupoids (Remark \ref{rem:LB-groupoids}). So, let $(L_{\mathcal G}  \rightrightarrows 0_M; \mathcal G  \rightrightarrows M)$ be an LB-groupoid with core $L$, so that $L_{\mathcal G} \simeq t^\ast L$. In the following, we will always understand the latter isomorphism. Denote by $A$ the Lie algebroid of $\mathcal G$. 

\begin{proposition}\label{prop:D_J^1}
Both $(D L_{\mathcal G}  \rightrightarrows DL ; \mathcal G  \rightrightarrows M)$ and $(\mathfrak J^1 L_{\mathcal G}  \rightrightarrows A^\ast \otimes L; \mathcal G  \rightrightarrows M)$ are VB-groupoids in a natural way.
\end{proposition}

\begin{proof}[Proof. (A sketch)]
We only sketch the proof. The straightforward details are left to the reader. For $(D L_{\mathcal G}  \rightrightarrows DL ; \mathcal G  \rightrightarrows M)$, the statement is contained in \cite[Proposition 4.10]{ETV2016}. Notice however, that the authors of \cite{ETV2016} use $s^\ast L$ instead of $L_{\mathcal G} = t^\ast L$. Actually, this choice is purely conventional as there is a canonical vector bundle isomorphism $\tau : s^\ast L \to t^\ast L$, given by $(g, \ell) \mapsto (g, g.\ell)$. Here, we limit ourselves to describe the structure maps of $D L_{\mathcal G} \rightrightarrows DL$, using our conventions. The source $s_D : DL_{\mathcal G} \to DL$ is $s_D = D (s_L \circ \tau^{-1})$. Here we denoted by $s_L : s^\ast L \to L$ the regular line bundle map induced by the source $s : \mathcal G \to M$. The target $t_D : DL_{\mathcal G} \to DL$ is $t_D = Dt$. The manifold $(DL_{\mathcal G})^{(2)}$ of composable arrows in $DL_{\mathcal G}$ is then a vector bundle over $\mathcal G^{(2)}$ canonically isomorphic to $D (m^\ast L_G)$, where $m : \mathcal G^{(2)} \to \mathcal G$ is the multiplication in $\mathcal G$. The isomorphism 
\[
D(m^\ast L_{\mathcal G}) \to (DL_{\mathcal G})^{(2)}
\]
maps $\square$ to the pair $(D \operatorname{pr}_1 (\square), D \operatorname{pr}_2 (\square_\tau))$, where $\operatorname{pr}_i : \mathcal G^{(2)} \to \mathcal G$ are the projections, and $\square_\tau \in D m^\ast s^\ast L$ is obtained from $\square$ identifying $L_{\mathcal G}$ with $s^\ast L$ via $\tau$. Here we are also identifying $m^\ast L_{\mathcal G}$ with $\operatorname{pr}_1^\ast L_{\mathcal G}$ using that $t \circ m = t \circ \operatorname{pr}_1$, and we are identifying $m^\ast s^\ast L$ with $\operatorname{pr}_2^\ast s^\ast L$ using that $t \circ m = s \circ \operatorname{pr}_2$. The unit is $1_D = D 1_L$, where we denoted by $1_L : L \to L_{\mathcal G}$ the regular line bundle map given by $\ell \mapsto (1, \ell)$. It follows that the inversion is $i_D = D(i_L \circ \tau^{-1})$, where $i_L : s^\ast L \to L_{\mathcal G}$ is the regular vector bundle map given by $(g, \ell) \mapsto (g^{-1}, \ell)$. The core of $(DL_{\mathcal G}  \rightrightarrows DL; \mathcal G  \rightrightarrows M)$ is (canonically isomorphic to) $A$. If $\alpha$ is a section of $A$, and $x \in M$, then the embedding $A \hookrightarrow DL_{\mathcal G}$ identifies point $\alpha_x$ of $A$, with the value at $1_x$ of derivation $\mathbb D_\alpha : L_{\mathcal G} \to L_{\mathcal G}$ given by the composition 
\[
\mathbb D_\alpha = \tau \circ \mathcal L_{\overrightarrow \alpha} \circ \tau^{-1}, 
\]
where $\mathcal L_{\overrightarrow \alpha}$ is the Lie derivative of a section of $s^\ast L$ along the right invariant vector field $\overrightarrow{\alpha}$ corresponding to $\alpha$. As $\overrightarrow{\alpha}$ is tangent to the source fibers, this is well-defined.

The groupoid structure on $\mathfrak J^1 L_{\mathcal G}$ can be obtained by \emph{duality}. Namely, the tensor product of two VB-groupoids is not a VB-groupoid in general. However the tensor product of a generic VB-groupoid and a VB-groupoid with trivial side bundle \emph{is} a VB-groupoid. It follows that $\mathfrak J^1 L_{\mathcal G} = \operatorname{Hom} (D L_{\mathcal G}, L_{\mathcal G}) = (DL_{\mathcal G})^\ast \otimes L_{\mathcal G}$ is a VB-groupoid. Its side bundle is $A^\ast \otimes L = \operatorname{Hom} (A, L)$ and the structure maps are the following. Source and target $s_{\mathfrak J^1}, t_{\mathfrak J^1} : \mathfrak J^1 L_{\mathcal G} \to A^\ast \otimes L$ are given by
\[
\langle s_{\mathfrak J^1} (\psi), a \rangle = - \langle \psi, m_D (0, i_D(\mathbb D_a)) \rangle, \quad
\langle t_{\mathfrak J^1} (\psi), b \rangle = \langle \psi, m_D(\mathbb D_b ,0) \rangle 
\]
for all $\psi \in \mathfrak J^1 L_{\mathcal G}$ and $a, b \in A$, where we identified the fiber of $L_{\mathcal G}$ over $g$ with $L_{t(g)}$. The multiplication $m_{\mathfrak J^1}$ is uniquely determined by
\[
\langle m_{\mathfrak J^1} (\psi_1, \psi_2), m_D (\Delta_1, \Delta_2) \rangle = \langle \psi_1, \Delta_1 \rangle \cdot \langle \psi_2, \Delta_2 \rangle = \langle \psi_1, \Delta_1 \rangle + g_1. \langle \psi_2, \Delta_2 \rangle ,
\]
for all composable pairs $(\psi_1, \psi_2)$ in $\mathfrak J^1 L_{\mathcal G}$, and all composable pairs $(\Delta_1, \Delta_2)$ in $DL_{\mathcal G}$ over $(g_1, g_2) \in \mathcal G^{(2)}$, where a dot ``$\cdot$'' denotes the multiplication in $L_{\mathcal G}$.
The unit $1_{\mathfrak J^1} : A^\ast \otimes L \to \mathfrak J^1 L_{\mathcal G}$ is given by
\[
\langle 1_{\mathfrak J^1} (\phi), \Delta \rangle = \langle \phi, \Delta - (1_D \circ s_D) (\Delta) \rangle,
\]
for all $\phi \in A^\ast \otimes L$, and $\Delta \in DL_{\mathcal G}$ (as $\Delta - (1_D \circ s_D) (\Delta) $ is in the core of $DL_{\mathcal G}$, this is well-defined). It follows that the inversion
$i_{\mathfrak J^1} : \mathfrak J^1 L_{\mathcal G} \to \mathfrak J^1 L_{\mathcal G}$ is given by
\[
\langle i_{\mathfrak J^1} \psi, \Delta \rangle = \langle \psi, i_D (\Delta) \rangle^{-1} =  - g^{-1}.  \langle \psi, i_D (\Delta) \rangle
\]
for all $\psi \in \mathfrak J^1 L_{\mathcal G}$ and all $\Delta \in DL_{\mathcal G}$ over $g \in \mathcal G$, where a superscript ${}^{-1}$ in the middle term denotes the inversion in $L_{\mathcal G}$.
\end{proof}

\begin{remark}
Notice that nor derivations, nor $1$-jets of a generic VB-groupoid form a groupoid. However derivations and 1-jets of a VB-groupoid with trivial side bundle, in particular of an LB-groupoid, do. \end{remark}

It follows from Proposition \ref{prop:D_J^1} that the fibered product
\[
\mathfrak T^{l,m}_L  := \underset{\text{$l$ times}}{\underbrace{DL_{\mathcal G}  \times_{\mathcal G } \cdots \times_{\mathcal G }  DL_{\mathcal G}   }} \times_{\mathcal G } \underset{\text{$m$ times}}{\underbrace{\mathfrak J^1 L_{\mathcal G}  \times_{\mathcal G } \cdots \times_{\mathcal G } \mathfrak J^1 L_{\mathcal G} }},
\]
 is a Lie groupoid as well. The manifold of objects is the fibered product
\[
(\mathfrak T^{l,m}_L)_0 :=  \underset{\text{$l$ times}}{\underbrace{DL  \times_{M } \cdots \times_{M} DL }} \times_{M } \underset{\text{$m$ times}}{\underbrace{(A^\ast \otimes L)  \times_{M } \cdots \times_{M} (A^\ast \otimes L)}},
\]
and the structure maps are induced component-wise from those of $DL_{\mathcal G} \rightrightarrows DL$ and $\mathfrak J^1 L_{\mathcal G} \rightrightarrows A^\ast \otimes L$. Additionally, diagram
\[
\xymatrix{\mathfrak T^{l,m}_L \ar[d] \ar@<+2pt>[r] \ar@<-2pt>[r]& (\mathfrak T^{l,m}_L)_0 \ar[d] \\
\mathcal G \ar@<+2pt>[r] \ar@<-2pt>[r] & M
}
\]
is a fibered groupoid (however, not a VB-groupoid). We are now ready to define \emph{multiplicative Atiyah tensors}. First of all, notice that an Atiyah $(l,m)$-tensor $\mathcal T$ on $L_{\mathcal G}$ can be regarded as a map
\[
f_{\mathcal T} : \mathfrak T^{l,m}_L \to L_{\mathcal G}
\]

\begin{definition}
An Atiyah $(l,m)$ tensor $\mathcal T$ on $L_{\mathcal G}$ is \emph{multiplicative} if $f_{\mathcal T}$ is a (fibered) groupoid map.
\end{definition}

\begin{proposition}\label{prop:mult_Atiyah}
An Atiyah $(l,m)$ tensor $\mathcal T$ on $L_{\mathcal G}$ is multiplicative iff its homogenization $\widetilde{\mathcal T}$ is a multiplicative tensor on the groupoid $\widetilde{\mathcal G}$.
\end{proposition}

\begin{proof}
First of all, straightforward computations show that diagrams
\[
\begin{array}{c}
\xymatrix{
T \widetilde{\mathcal G} \ar[r]^-{p_T} \ar[d] & DL_{\mathcal G} \ar[d] \\
\widetilde{\mathcal G} \ar[r]^-p & \mathcal G
}
\end{array}, \quad \text{and} \quad
\begin{array}{c}
\xymatrix{
T^\ast \widetilde{\mathcal G} \ar[r]^-{p_{T^\ast}} \ar[d] & \mathfrak J^1L_{\mathcal G} \ar[d] \\
\widetilde{\mathcal G} \ar[r]^-p & \mathcal G
}
\end{array}
\]
are VB-groupoid maps, i.e.~they are simultaneously vector bundle maps and Lie groupoid maps. Now let $\mathcal T$ and $\widetilde{\mathcal T}$ be as in the statement and consider the corresponding smooth maps
\[
f_{\mathcal T} : \mathfrak T^{l,m}_L \to L_{\mathcal G}, \quad \text{and} \quad f_{\widetilde{\mathcal T}} : \mathfrak T^{l,m}\widetilde{\mathcal G} \to \mathbb{R}_{\widetilde{\mathcal G}}.
\]
Regardless whether or not $\mathcal T$ or $\widetilde{\mathcal T}$ is multiplicative, both $f_{\mathcal T}$ and $f_{\widetilde{\mathcal T}}$ preserve automatically composability. Hence, it is enough to check that $f_{\mathcal T}$ preserves multiplication iff so does $f_{\widetilde{\mathcal T}}$. So, let $\epsilon, \zeta \in \widetilde{\mathcal G}$ be composable arrows, and take composable
\[
(v_1, \ldots, v_l; \eta_1, \ldots, \eta_m), (w_1, \ldots, w_l; \theta_1, \ldots, \theta_m) \in \mathfrak T^{l,m} \widetilde{\mathcal G}
\]
over them, meaning that $v_i,w_i$ and $\eta_j, \theta_j$ are composable for each $i$ and $j$. Additionally, put $\Delta_i = p_T(v_i)$, $\square_i = p_T (w_i)$, $\psi_j = p_{T^\ast} (\eta_j)$, and $\psi_j = p_{T^\ast} (\theta_j)$, so that 
\[
\begin{aligned}
& f_{\widetilde{\mathcal T}} ((v_1, \ldots, v_l; \eta_1, \ldots, \eta_m) \cdot (w_1, \ldots, w_l; \theta_1, \ldots, \theta_m))  \\ 
& = f_{\widetilde{\mathcal T}} (v_1 \cdot w_1, \ldots, v_l \cdot w_l; \eta_1 \cdot \theta_1, \ldots, \eta_m \cdot \theta_m) \\
& = \langle \epsilon \cdot \zeta, f_{\mathcal T} (p_T(v_1 \cdot w_1), \ldots, p_T(v_l \cdot w_l); p_{T^\ast}(\eta_1 \cdot \theta_1), \ldots, p_{T^\ast}(\eta_m \cdot \theta_m)) \rangle
\\
& = \langle \epsilon \cdot \zeta, f_{\mathcal T} (\Delta_1 \cdot \square_1, \ldots, \Delta_l \cdot \square_l; \chi_1 \cdot \psi_1, \ldots, \chi_m \cdot \psi_m) \rangle
\end{aligned} 
\]
Now $f_{\widetilde{\mathcal T}}$ preserves multiplication iff
\[
\begin{aligned}
&  f_{\widetilde{\mathcal T}} ((v_1, \ldots, v_l; \eta_1, \ldots, \eta_m) \cdot (w_1, \ldots, w_l; \theta_1, \ldots, \theta_m)) \\
& = f_{\widetilde{\mathcal T}} (v_1, \ldots, v_l; \eta_1, \ldots, \eta_m) + f_{\widetilde{\mathcal T}} (w_1, \ldots, w_l; \theta_1, \ldots, \theta_m)) \\
& = \langle \epsilon, f_{\mathcal T} (\Delta_1, \ldots, \Delta_l; \chi_1, \ldots, \chi_m) \rangle 
 + \langle \zeta,  f_{\mathcal T} (\square_1, \ldots, \square_l; \psi_1, \ldots, \psi_m)) \rangle\\
& = \langle \epsilon \cdot \zeta,  f_{\mathcal T} (\Delta_1, \ldots, \Delta_l; \chi_1, \ldots, \chi_m) + p(\epsilon).  f_{\mathcal T} (\square_1, \ldots, \square_l; \psi_1, \ldots, \psi_m)) \rangle,
\end{aligned}
\]
and the claim follows from surjectivity of $p_T, p_{T^\ast}$.
\end{proof}

\begin{example}[Multiplicative contact structures as multiplicative symplectic Atiyah forms]\label{ex:mult_contact}
A distribution $\mathcal D \subset T \mathcal G$ on a Lie groupoid $\mathcal G \rightrightarrows M$ is \emph{multiplicative} if it is a VB-subgroupoid of the tangent VB-groupoid $(T \mathcal G  \rightrightarrows TM; \mathcal G \rightrightarrows M)$. In particular, there is a distribution $\mathcal D_0 \subset TM$ on $M$ so that $(\mathcal D \rightrightarrows \mathcal D_0; \mathcal G \rightrightarrows M)$ is a VB-groupoid, and we say that $\mathcal D$ \emph{covers} $\mathcal D_0$. Denote by $\nu (\mathcal D) = T\mathcal G / \mathcal D$ and $\nu (\mathcal D_0) = TM/ \mathcal D_0$ the normal bundles to $\mathcal D$ and $\mathcal D_0$ respectively. Then $(\nu (\mathcal D)  \rightrightarrows \nu (\mathcal D_0); \mathcal G  \rightrightarrows M)$, with the obvious structure maps, is a VB-groupoid as well, and projection $T\mathcal G \to \nu (\mathcal D)$ is a VB-groupoid map.

Now, a \emph{contact groupoid} is a Lie groupoid $\mathcal G \rightrightarrows M$ equipped with a \emph{multiplicative contact structure}, i.e.~a multiplicative contact distribution $H \subset T \mathcal G$ with the additional property that $H$ covers $TM$. In other words, the normal bundle $\nu (H)$ sits in an LB-groupoid $(\nu (H)  \rightrightarrows 0_M; \mathcal G  \rightrightarrows M)$. In the following, we denote by $L_{\mathcal G}$ the normal bundle $\nu (H)$ of a multiplicative contact structure $H$, and by $L$ its core, so that, in particular, $\mathcal G$ acts on $L$, and $L_{\mathcal G} = t^\ast L$ as VB-groupoids over $\mathcal G$. 

Now, let $(L_{\mathcal G}  \rightrightarrows 0_M; \mathcal G  \rightrightarrows M)$ be an LB-groupoid, and let $H$ be a contact structure such that $\nu (H) = L_{\mathcal G}$. Consider the symplectic Atiyah $2$-form $\omega \in \Omega^2_{L_{\mathcal G}}$ corresponding to $H$. Then $H$ is a multiplicative contact structure iff $\omega$ is a multiplicative Atiyah $2$-form. To see this we argue as follows.

First of all, Crainic and Salazar \cite{CS2015} (see also \cite{CSS2015}) shows that $H$ is a multiplicative contact structure iff its structure form $\theta \in \Omega^1 (M, t^\ast L)$ is \emph{multiplicative}, in the sense that
\[
(m^\ast \theta)_{(g, h)} = \operatorname{pr}_1^\ast \theta_{g} + g . \operatorname{pr}_2^\ast \theta_{h}.
\]
for all $(g, h) \in \mathcal G^{(2)}$. When $L = \mathbb R_M$ is the trivial line bundle, equipped with the trivial representation, one recovers definition (\ref{eq:mult_form}) of a multiplicative $1$-form. It is also easy to see that $\theta : T\mathcal G \to t^\ast L$ is multiplicative iff it is a VB-groupoid map. Now, put $\Theta := \theta \circ \sigma \in 
\Omega^1_{L_{\mathcal G}}$, and recall that $\omega = d_D \Theta$. Then $\omega$ is multiplicative
 iff $\Theta$ is so. To see this, we pass to the homogenizations $\widetilde \omega$ and
  $\widetilde \Theta$. Then $\widetilde \omega = d \widetilde \Theta$, and $\widetilde 
  \Theta = i_{\mathcal Z} \widetilde \omega$, where $\mathcal Z$ is the Euler vector field
   on $\widetilde{\mathcal G}$. From Remark \ref{rem:Euler_mult}, $\mathcal Z$ is a
    multiplicative vector field. As exterior differential and contraction with a multiplicative
     vector field preserve multiplicative forms, it immediately follows that $\widetilde 
     \omega$ is multiplicative iff $\widetilde \Theta$ is so. It remains to prove that $\theta$
      is multiplicative iff $\Theta$, or, equivalently, $\widetilde \Theta$, is so. To do this, identify $\widetilde{\mathcal G}$ with the pull-back bundle $t^\ast \widetilde M = t^\ast L^\ast \smallsetminus 0$, denote by $p : \widetilde{\mathcal G} \to \mathcal G$ the projection, and notice that, from the definition of $\widetilde \Theta$,
      \[
      \widetilde \Theta_{(g, \epsilon)} (v) = \langle \epsilon, \theta_{g} (p_{\ast} (v)) \rangle,
      \]
     for all $(g,\epsilon) \in \widetilde{\mathcal G}$, i.e.~$g \in \mathcal G$ and $\epsilon \in L^\ast_{t(g)} \smallsetminus 0$, and all $v \in T_{\epsilon} \widetilde{\mathcal G}$.
     Now, take $((g, \epsilon), (h, \eta)) \in \widetilde{\mathcal G}{}^{(2)}$. This means that $(g,h) \in \mathcal G^{(2)}$, and $\epsilon = h.\eta$ (and, in this case, $(g, \epsilon) \cdot
       (h, \eta) = (gh, \eta)$). For all $(v, w) \in T \widetilde{\mathcal G}{}^{(2)}$ we have
\[
\begin{aligned}
& (m^\ast \widetilde \Theta - \operatorname{pr}_1^\ast \widetilde \Theta - \operatorname{pr}_2^\ast \widetilde \Theta)_{((g, \epsilon), (h, \eta))} (v,w) \\
& = \widetilde \Theta_{(gh, \eta)} (v \cdot w) - \widetilde \Theta_{(g, \epsilon)}(v) - \widetilde \Theta_{(h, \eta)}(w) \\
& = \langle \eta, \theta_{gh} (p_\ast (v\cdot w)) - \theta_h (p_\ast (w)) \rangle - \langle \epsilon, \theta_g(p_\ast (v)) \rangle \\
& =  \langle \eta, \theta_{gh} (p_\ast (v) \cdot p_\ast (w)) - \theta_h (p_\ast (w)) \rangle - \langle h.\eta, \theta_g(p_\ast (v)) \rangle \\
& = \langle \eta, \theta_{gh} (p_\ast (v) \cdot p_\ast (w)) - \theta_h (p_\ast (w)) \rangle - \langle \eta, h.\theta_g(p_\ast (v)) \rangle \\
& =  \langle \eta, \theta_{gh} (p_\ast (v) \cdot p_\ast (w)) - \theta_h (p_\ast (w)) - h.\theta_g(p_\ast (v)) \rangle \\
& =  \langle \eta, ((m^\ast \theta)_{(g,h)}  -  \operatorname{pr}_1^\ast \theta_g  - h. \operatorname{pr}_2^\ast  \theta_h) (p_\ast(v), p_\ast(w)) \rangle.
\end{aligned}
\] 
As $p_\ast$ is surjective, and $\eta$ is arbitrary, we immediately see that $\Theta$ is multiplicative iff $\theta$ is so.

We conclude that a contact groupoid is basically the same as an LB-groupoid equipped with a multiplicative symplectic Atiyah $2$-form.

\end{example}

   	%\subsection{hereditarity of multiplicativity under homogenization}
	%\subsection{Homogenization of line bundle algebroids}
	%\subsection{The complex case}

\section{Integration of Poisson structures}

\subsection{Integration of Jacobi manifolds}

Recall from Example \ref{ex:Jacobi_structures} that a \emph{Jacobi manifold} is a manifold $M$ equipped with a \emph{Jacobi structure}, i.e.~a pair $(L, \{-,-\})$ consisting of a line bundle $L \to M$ and a Lie bracket $\{-,-\} : \Gamma (L) \times \Gamma (L) \to \Gamma (L)$ which is a bi-derivation. Bracket $\{-,-\}$ is called the \emph{Jacobi bracket}, and $L \to M$, equipped with the Jacobi bracket, is called the \emph{Jacobi bundle}. A \emph{Jacobi map} between Jacobi manifolds $(M_1, L_1, \{-,-\}_1)$ and $(M_1, L_2, \{-,-\}_2)$ is a (regular) line bundle map $F : L_1 \to L_2$ such that $\{F^\ast \lambda, F^\ast \mu\}_1 = F^\ast \{ \lambda, \mu\}_2$ for all $\lambda, \lambda' \in \Gamma (L_2)$. Jacobi manifolds encompass contact, locally conformally simplectic, and Poisson manifolds as special instances. Notice that a Jacobi bracket on $L$ is the same as a skew-symmetric Atiyah $(0,2)$-tensor $J : \wedge^2 \mathfrak J^1 L \to L$ satisfying the integrability condition 
\[
[J,J]^{SJ} = 0,
\] 
where $[-,-]^{SJ}$ is the Schouten-Jacobi bracket, and we write $\{-,-\} \equiv J$. Every Jacobi manifold $(M, L, \{-,-\} \equiv J)$ determines a Lie algebroid structure, denoted $(\mathfrak J^1 L)_J$, and called the \emph{jet algebroid} of $(M, L, J)$, on the first jet bundle $\mathfrak J^1 L$. The Lie bracket $[-,-]_J$ on sections of $(\mathfrak J^1 L)_J$ is given by 
\[
[\psi, \chi]_J = \mathcal L_{J^\sharp \psi} \chi - \mathcal L_{J^\sharp \chi} \psi - d_D \langle J, \psi \wedge \chi \rangle
\]
for all $\psi, \chi \in \Gamma (\mathfrak J^1 L)$, where $J^\sharp : \mathfrak J^1 L \to DL$ is the vector bundle map induced by $J$ in the obvious way. The anchor $\rho_J : (\mathfrak J^1 L)_J \to TM$ is given by the composition $\sigma \circ J^\sharp$. Additionally, the jet algebroid $(\mathfrak J^1 L)_J$ acts on the Jacobi bundle $L$ in a canonical way. The flat connection defining the action is $J^\sharp$.

Crainic and Salazar \cite{CS2015} proved that \emph{Jacobi manifolds integrate to contact groupoids} in the following sense (see also \cite{KS1993} where only partial results were obtained, and \cite{CZ2007} where only the case when $L = \mathbb R_M$ is the trivial line bundle is considered; foundational results on this topic can be found in the older references \cite{L1993,D1997}). First of all a Jacobi manifold $(M, L, J)$ is said to be \emph{integrable} if the associated jet groupoid is integrable, i.e.~it is isomorphic to the Lie algebroid of a Lie groupoid. We then have the following 

\begin{theorem}\label{theor:int_Jacobi}
\quad
\begin{enumerate}
\item Let $(\mathcal G \rightrightarrows M, H)$ be a contact groupoid with $\nu (H) = L_{\mathcal G} $. Then $M$ is equipped with a unique Jacobi structure $(L, J)$, such that $L$ is the core of the LB-groupoid $(L_{\mathcal G} \rightrightarrows 0_M; \mathcal G \rightrightarrows M)$, and the natural bundle map $t_L : L_{\mathcal G} \simeq t^\ast L \to L$ is a Jacobi map (covering the target $t : \mathcal G \to M$). The Jacobi manifold $(M, L, J)$ is called the differentiation of $(\mathcal G, H)$. 
\item Conversely, let $(M, L, J)$ be an integrable Jacobi manifold and let $\mathcal G$ be the source-simply connected integration of the jet algebroid $(\mathfrak J^1 L)_J$. Then $\mathcal G$ is equipped with a unique multiplicative contact structure $H$ inducing precisely the Jacobi structure $(L, J)$ on $M$ by differentiation. The contact groupoid $(\mathcal G, H)$ is called the integration of $(M, L, J)$.
\end{enumerate}
Integration and differentiation establish a one-to-one correspondence between integrable Jacobi manifolds and source-simply connected contact groupoids. Shortly, Jacobi manifolds integrate to contact groupoids.
\end{theorem}

Theorem \ref{theor:int_Jacobi} was first proved by Crainic and Zhu in \cite{CZ2007}, via the homogenization scheme. However, they only consider the case when the Jacobi bundle is trivial. Later Crainic and Salazar reproved Theorem \ref{theor:int_Jacobi} using the so called \emph{Spencer operators} \cite{CS2015,CSS2015}, without any reference to homogenization.

In this paper we provide a proof of the general case exploiting homogenization, thus filling the gap between the proofs of \cite{CZ2007} and \cite{CS2015} (Theorem \ref{theor:Jac_hom_Poiss}). However, our main concern is generalizing Theorem \ref{theor:int_Jacobi} to the holomorphic setting. This is done in the last Section \ref{Sec:Hol_Jac}. Also in this case we use the homogenization scheme. The homogenization of a holomorphic Jacobi manifold is a homogeneous holomorphic Poisson manifold. In particular, we need to discuss, in some details, the integration problem for homogeneous holomorphic Poisson manifolds. This is done in the present section by stages. We first need to discuss integration of Poisson manifolds equipped with additional compatible structures, namely Nijenhuis (more specifically complex) structures, homogeneity vector fields, or both. Accordingly, we provide a hierarchy of integration theorems. Some of those already appeared in the literature. We provide (partially) new proofs relying on the language of Spencer operators \cite{CSS2015}, which we think is particularly efficient when working with differential forms on Lie groupoids. 

\subsection{Spencer operators}

We begin recalling what are Spencer operators. Conceptually they are \emph{infinitesimal data} associated to multiplicative (vector valued) forms on Lie groupoids. We will only need \emph{scalar valued} Spencer operators. In this case, the original definition \cite{CSS2015} boils down to the following. Let $k$ denote a non-negative integer. A (scalar valued) \emph{$k$-Spencer operator} on a Lie algebroid $A \Rightarrow M$ is a pair $(\mathcal D, \ell)$ consisting of
\begin{itemize}
\item[$\triangleright$] a first order differential operator $\mathcal D : \Gamma (A) \to \Omega^k (M)$, and 
\item[$\triangleright$] a vector bundle map $\ell : A \to \wedge^{k-1} T^\ast M$
\end{itemize}
such that
\[
\mathcal D (f \alpha) = f \mathcal D (\alpha) + df \wedge \ell (\alpha),
\]
and, additionally,
\begin{equation}\label{eq:Spencer}
\begin{aligned}
\mathcal L_{\rho (\alpha)} \mathcal D(\beta) - \mathcal L_{\rho (\beta)} \mathcal D (\alpha) - \mathcal D ([\alpha, \beta]) & = 0 \\
\mathcal L_{\rho (\alpha)} \ell (\beta) + i_{\rho (\beta)} \mathcal D (\alpha) - \ell ([\alpha, \beta])& = 0 \\
i_{\rho (\alpha)} \ell (\beta) + i_{\rho (\beta)} \ell (\alpha) & = 0
\end{aligned}
\end{equation}
for all $\alpha, \beta \in \Gamma (A)$.
Notice that this definition is equivalent to the earlier definition of \emph{IM differential form} \cite{BC2012} (but not exactly the same). The relevance of Spencer operators resides in the following

\begin{theorem}[Crainic, Salazar, Struchiner {\cite{CSS2015}}]
\quad
\begin{enumerate}
\item Let $\mathcal G \rightrightarrows M$ be a Lie groupoid with Lie algebroid $A$, and let $\omega \in \Omega^k (\mathcal G)$ be a multiplicative $k$-form on $\mathcal G$. Then $A$ is equipped with a canonical $k$-Spencer operator $(\mathcal D, \ell)$, called the differentiation of $\omega$, and given by 
\begin{equation}\label{eq:Spencer_groupoid}
\begin{aligned}
\mathcal D (\alpha) & = 1^\ast (\mathcal L_{\overrightarrow \alpha} \omega) \\
\ell (\alpha) & = 1^\ast (i_{\overrightarrow \alpha} \omega)
\end{aligned}
\end{equation}
for all $\alpha \in \Gamma (A)$, where $1 : M \to \mathcal G$ is the unit, and $\overrightarrow \alpha$ is the right invariant vector field on $\mathcal G$ corresponding to $\alpha$.
\item Conversely, let $A \Rightarrow M$ be an integrable Lie algebroid equipped with a $k$-Spencer operator $(\mathcal D, \ell)$, and let $\mathcal G$ be its source-simply connected integration. Then $\mathcal G$ is equipped with a unique multiplicative $k$-form $\alpha$, called the integration of $(\mathcal D, \ell)$, and satisfying (\ref{eq:Spencer_groupoid}).
\end{enumerate}

Integration and differentiation establish a one-to-one correspondence between integrable Lie algebroids equipped with a $k$-Spencer operator and source-simply connected Lie groupoids equipped with a multplicative $k$-form.
\end{theorem}

\begin{example}\label{ex:symplectic_groupoid}
It is well known that Poisson manifolds integrate to \emph{symplectic groupoids}, i.e.~groupoids equipped with a multiplicative symplectic form. Recall from Example \ref{ex:Nijenhuis} that a Poisson manifold $(M, \pi)$ determines a \emph{cotangent Lie algebroid} $(T^\ast M)_\pi$, i.e.~a Lie algebroid structure on $T^\ast M$. The Lie bracket $[-,-]_\pi$ on sections $\Omega^1 (M)$ of $(T^\ast M)_\pi$ is given by Equation (\ref{eq:brack_pi}), and the anchor $(T^\ast M)_\pi \to TM$ is the vector bundle map $\pi^\sharp : T^\ast M \to TM$ induced by $\pi$. A Poisson manifold is said to be \emph{integrable} if its cotangent algebroid is so. The following \emph{integration results} hold for Poisson manifolds.

If $(\mathcal G \rightrightarrows M, \omega)$ is a symplectic groupoid, then $M$ is canonically equipped with a Poisson structure $\pi$. Conversely, if $(M, \pi)$ is an integrable Poisson manifold and $\mathcal G$ is a source-simply connected integration of $(T^\ast M)_\pi$, then $\mathcal G$ is canonically equipped with a multiplicative symplectic structure $\omega$. These constructions establish a one-to-one correspondence between integrable Poisson manifolds and source-simply connected symplectic groupoids. If $(\mathcal D, \ell)$ is the Spencer operator on $(T^\ast M)_\pi$ corresponding to $\omega$, then 
\[
\mathcal D = d : \Omega^1 (M) \to \Omega^2 (M)
\]
is the de Rham differential, and
\[
\ell = \mathbb 1 : T^\ast M \to T^\ast M
\]
is the identity. We will often use these remarks in what follows.
\end{example}

    \subsection{Poisson-Nijenhuis}
    We now pass to the integration problem for Poisson-Nijenhuis manifolds. This has been first discussed in \cite{SX2007}. Recall that a \emph{Poisson-Nijenhuis manifold} is a manifold $M$ equipped with a \emph{Poisson-Nijenhuis structure}, i.e.~a pair $(\pi, N)$ consisting of a bivector $\pi$ and a $(1,1)$-tensor $N : TM \to TM$ such that
\begin{enumerate}
\item $\pi$ is Poisson,
\item $N$ is Nijenhuis (i.e.~its torsion vanishes),
\item the following formulas hold:
\begin{align}
\pi (\alpha, N^\ast \beta) + \pi (\beta, N^\ast \alpha) & = 0  \label{eq:1}\\
\mathcal L_{\pi^\sharp \alpha} N^\ast \beta - \mathcal L_{\pi^\sharp \beta} N^\ast \alpha - d \pi (N^\ast \alpha, \beta) & = N^\ast [\alpha, \beta]_\pi \label{eq:2}
\end{align}
\end{enumerate}
where $N^\ast : T^\ast M \to T^\ast M$ is the transpose of $N$.

\begin{remark}
Equations (\ref{eq:1}) and (\ref{eq:2}) can be actually written in several equivalent ways. The present one is particularly convenient in our setting. Indeed it precisely says that $(d \circ N^\ast, N^\ast)$ is a \emph{Spencer operator} for the cotangent algebroid $(T^\ast M)_\pi$.
\end{remark}

From (\ref{eq:1}), tensor $\pi_N = \pi (N -,-)$ is actually a bi-vector. It then follows from (\ref{eq:2}) and the other axioms, that $\pi_N$ is actually a Poisson bivector.

Let $(\pi, N)$ be a Poisson-Nijenhuis structure. If $\pi$ is non degenerate we denote by $\Omega$ the corresponding symplectic form, and call $(\Omega, N)$ a \emph{symplectic-Nijenhuis structure}. In this case tensor $\Omega_N = \Omega (N -, -)$ is actually a (non-necessarily non-degenerate) closed $2$-form.

\begin{lemma}[Magri, Morosi {\cite{MM1984}}]\label{lem:Magri-Morosi}
Let $\pi, N$ be a Poisson bivector, and a $(1,1)$-tensor respectively. Assume that $\pi$ is non-degenerate, so $\pi = \Omega^{-1}$ for some symplectic form $\Omega$. Assume also that Equation (\ref{eq:1}) holds, so that both $\Omega_N := \Omega (N-,-)$, and $\Omega_{N^2} = \Omega (N^2-,-)$ are $2$-forms. Then the following conditions are equivalent
\begin{enumerate}
\item $\Omega_N$ and $\Omega_{N^2}$ are closed;
\item the torsion of $N$ vanishes and Equation (\ref{eq:2}) holds, i.e.~$(\pi, N)$ is a Poisson-Nijenhuis structure.
\end{enumerate}
\end{lemma}

\begin{remark}
It follows from Lemma \ref{lem:Magri-Morosi} that a symplectic-Nijenhuis structure $(\Omega, N)$ is equivalent to a pair $(\Omega, \Omega')$ consisting of a symplectic form $\Omega$ and a closed $2$-form $\Omega'$ such that $\Omega_{N}' = \Omega_{N^2} = \Omega (N^2-, -)$ is also a closed $2$-form, where $N = \Omega^\sharp \circ \Omega'_\flat$. In particular $\Omega' = \Omega_N$.
\end{remark}

\begin{definition}
A \emph{symplectic-Nijenhuis groupoid} is a Lie groupoid $\mathcal G \rightrightarrows M$ equipped with a \emph{multiplicative symplectic-Nijenhuis structure}, i.e.~a symplectic-Nijenhuis structure $(\Omega, \mathcal N)$ such that both $\Omega$ and $\mathcal N$ are multiplicative.
\end{definition}

\begin{remark}
Recall that multiplicativity of $\Omega$ is equivalent to the induced vector bundle map $\Omega_\flat : T\mathcal G \to T^\ast \mathcal G$ being a VB-groupoid map (Remark \ref{rem:mult_tens}). Similarly, multiplicativity of $\mathcal N$ means that $\mathcal N : T \mathcal G \to T \mathcal G$ is a VB-groupoid map. It follows that $\Omega_\flat \circ \mathcal N$ and $\Omega_\flat \circ \mathcal N^2$ are also VB-groupoid map, i.e.~$\Omega_{\mathcal N}, \Omega_{\mathcal N^2}$ are multiplicative (closed) $2$-forms.
\end{remark}

\begin{proposition}\label{prop:PN}
\quad
\begin{enumerate}
\item Let $(\mathcal G \rightrightarrows M, \Omega, \mathcal N)$ be a symplectic-Nijenhuis groupoid. Then $M$ is equipped with a unique Poisson-Nijenhuis structure $(\pi, N)$ such that the target $t : \mathcal G \to M$ is a Poisson map, and $N = \mathcal N|_M$. The Poisson-Nijenhuis manifold $(M, \pi, N)$ is called the differentiation of $(\mathcal G, \Omega, \mathcal N)$.
\item Conversely, let $(M, \pi, N)$ be an \emph{integrable} Poisson-Nijenhuis manifold, i.e.~the cotangent algebroid $(T^\ast M)_\pi$ is integrable, and let $\mathcal G$ be its source-simply connected groupoid. Then $\mathcal G$ is equipped with a unique multiplicative symplectic-Nijenhuis structure $(\Omega, \mathcal N)$ inducing precisely the Poisson-Nijenhuis structure $(\pi, N)$ on $M$ by differentiation. The symplectic-Nijenhuis groupoid $(\mathcal G, \Omega, \mathcal N)$ is called the integration of $(M, \pi, N)$.
\end{enumerate} 
Integration and differentiation establish a one-to-one correspondence between integrable Poisson-Nijenhuis manifolds and source-simply connected symplectic-Nijenhuis groupoids. Shortly, Poisson-Nijenhuis manifolds integrate to symplectic-Nijenhuis groupoids.
\end{proposition}

\begin{proof}
(A version of) the theorem has been proven in \cite{SX2007} using the \emph{Universal Lifting Theorem} of \cite{ISX2012}. Using Lemma \ref{lem:Magri-Morosi} we provide an alternative proof via Spencer operators. Our proof is essentially the same as the one recently appeared in a local groupoid setting \cite{CMS2017b} (see also \cite{CMS2017a}). We report it here for completeness. So, let $(\mathcal G \rightrightarrows M, \Omega, \mathcal N)$ be a symplectic-Nijenhuis groupoid. Denote by $A$ the Lie algebroid of $\mathcal G$, and let $(\mathcal D, \ell)$ be the Spencer operator that differentiates $\Omega$. As $\Omega$ is non-degenerate, $\ell : A \to T^\ast M$ is an isomorphism, and we use it to identify $A$ with the cotangent algebroid $(T^\ast M)_\pi$ of the unique Poisson structure on $M$ such that the target is a Poisson map. From $d \Omega = 0$ it follows $\mathcal D = d \circ \ell$ (see Example \ref{ex:symplectic_groupoid}). Now consider the $2$-form $\Omega_{\mathcal N}$ and let $(\mathcal D_{\mathcal N}, \ell_{\mathcal N})$ be the corresponding Spencer operator. From $d\Omega_{\mathcal N} = 0$ and (\ref{eq:Spencer_groupoid}) it follows that $\mathcal D_{\mathcal N} = d \circ \ell_{\mathcal N}$. Let $N : TM \to TM$ be the transpose of $\ell_{\mathcal N} : T^\ast M \to T^\ast M$. We claim that $(\pi, N)$ is a Poisson-Nijenhuis structure. Namely, Equations (\ref{eq:1}) and (\ref{eq:2}) are equivalent to $(\mathcal D_{\mathcal N}, \ell_{\mathcal N})$ being a Spencer operator, and it remains to check that $N$ is torsion-free. To see this, notice, first of all, that, from multiplicativity, $\mathcal N$ is tangent to $M$, i.e.~it maps tangent vectors to $M$ to tangent vectors to $M$ and we denote by $\mathcal N|_M : TM \to TM$ the restriction. It immediately follows that the torsion of $\mathcal N$ is also tangent to $M$, and its restriction is the torsion of $\mathcal N |_M$. Hence it suffices to show that $\mathcal N|_M = N$. So, let $\alpha \in \Omega^1 (M)$, and $v \in TM$, and compute
\[
\begin{aligned}
\langle \alpha, N v \rangle & = \langle \ell_{\mathcal N} (\alpha), v \rangle
                                        = \langle 1^\ast (i_{\overrightarrow \alpha} \Omega_{\mathcal N}), v \rangle 
                                          = \Omega_{\mathcal N} (\overrightarrow \alpha, v) \\
                                          & = \Omega (\overrightarrow \alpha, \mathcal N v) 
                                           = \langle 1^\ast (i_{\overrightarrow \alpha} \Omega), \mathcal N v \rangle
                                           = \langle \alpha, \mathcal N v \rangle.
\end{aligned}
\]
This concludes the proof of the first part of the statement. 

For the second part, let $(M, \pi, N)$ be an integrable Poisson-Nijenhuis manifold, let $\mathcal G$ be the source-simply connected integration of $(T^\ast M)_\pi$, and let $\Omega$ be the multiplicative symplectic form on $\mathcal G$ such that the target is a Poisson map. Recall again that, in view of (\ref{eq:1}) and (\ref{eq:2}), $(d \circ N^\ast, N^\ast)$ is a Spencer operator on $(T^\ast M)_\pi$. So it induces a closed multiplicative form $\Omega'$ on $\mathcal G$ via integration. Put $\mathcal N = \Omega^\sharp \circ \Omega'_\flat$. We want to show that $(\mathcal G, \Omega, \mathcal N)$ is a symplectic-Nijenhuis groupoid. From multiplicativity of $\Omega$ and $\Omega'$, $\mathcal N$ is multiplicative as well. Additionally $\Omega_{\mathcal N} = \Omega'$ is a closed $2$-form by construction. From Lemma \ref{lem:Magri-Morosi}, it remains to check that $\Omega'_{\mathcal N} = \Omega_{\mathcal N^2} = \Omega (\mathcal N^2 -, -)$ is closed as well. It suffices to check that the Spencer operator $(\mathcal D'', \ell'')$ that differentiates $d\Omega_{\mathcal N^2}$ vanishes. As $d\Omega_{\mathcal N^2}$ is closed, it is enough to check that $\ell'' = 0$. Recall from \cite[Equation (B.3.12)]{MM1984} that, for a closed $2$-form $\Omega$ and a $(1,1)$-tensor $\mathcal N$ such that $\Omega (\mathcal N v, w) = \Omega (v, \mathcal N w)$, we have
\[
i_v i_w d \Omega_{\mathcal N^2} = - i_{\mathcal T_{\mathcal N}(v, w)} \Omega,
\] 
for all tangent vectors $v,w$. Now, exactly as above, it follows from multiplicativity of $\mathcal N$, that $\mathcal T_{\mathcal N}$ is tangent to $M$, and it agrees with $\mathcal T_N$ on $M$, hence $\mathcal T_{\mathcal N}$ vanishes on $M$. So
\[
\begin{aligned}
\langle \ell'' (\overrightarrow \alpha), v \wedge w \rangle & = \langle 1^\ast (i_{\overrightarrow \alpha} d \Omega_{\mathcal N^2}), v \wedge w \rangle = i_{\overrightarrow \alpha} i_v i_w d \Omega_{\mathcal N^2}
= - i_{\overrightarrow \alpha} i_{\mathcal T_{\mathcal N}(v, w)} \Omega \\
& = i_{\mathcal T_{\mathcal N}(v, w)}  i_{\overrightarrow \alpha} \Omega = \langle 1^\ast (i_{\overrightarrow \alpha} \Omega), \mathcal T_{\mathcal N}(v, w) \rangle
= \langle \alpha , \mathcal T_{\mathcal N}(v, w) \rangle = 0
\end{aligned}
\]
for all $\alpha \in \Omega^1 (M)$, and $v,w \in TM$.
\end{proof}

 \subsection{Holomorphic Poisson}
 \begin{remark}\label{rem:complex}
Let $N$ and $\mathcal N$ be as in Proposition \ref{prop:PN}. Then $N$ is a complex structure, i.e.~$N^2 = - \mathbb 1$, iff $\mathcal N$ is so. Indeed, being $N$ the restriction of $\mathcal N$ to $M$, it is clear that if $\mathcal N$ is a complex structure, so is $N$. Conversely, let $N$ be a complex structure, and compute the Spencer operator $(d \circ \ell, \ell)$ of the closed $2$-form $\Omega_{\mathcal N^2}$. We have
\[
\langle \ell (\alpha), v\rangle = \langle 1^\ast (i_{\overrightarrow \alpha}\Omega_{\mathcal N^2} ), v \rangle = \Omega ( \overrightarrow \alpha, N^2 v ) = \langle 1^\ast (i_{\overrightarrow \alpha}\Omega), N^2 v \rangle = - \langle \alpha, v\rangle, 
\]
for all $\alpha \in \Omega^1 (M)$ and $v \in TM$, and we can conclude that $\Omega_{\mathcal N^2} = -\Omega$ so that $\mathcal N^2 = - \mathbb 1$.
\end{remark}

We can put Proposition \ref{prop:PN} and Remark \ref{rem:complex} together to reprove the integration theorem for holomorphic Poisson manifolds. Recall that a \emph{holomorphic Poisson manifold} is a complex manifold $X = (M, j)$ equipped with a holomorphic Poisson bivector, i.e.~a bivector $\Pi \in \Gamma (\wedge^2 T^{1,0}X)$ such that $\overline \partial \Pi = 0$ ($\Pi$ is \emph{holomorphic}) and $[\Pi, \Pi]^S = 0$ ($\Pi$ is Poisson). From the real differential geometry point of view, a holomorphic Poisson manifold $(M, j, \Pi)$ is equivalent to a Poisson-Nijenhuis manifold $(M, \pi, N)$ such that the Nijenhuis tensor $N$ is an almost complex (hence complex) structure. Under this equivalence $(j, \Pi)$ corresponds to $(\operatorname{Re} \Pi, j)$. Under the inverse equivalence $(\pi, N)$ corresponds to $(N, \pi - i \pi_N)$. Additionally $\Pi$ is non-degenerate, hence it comes from a complex symplectic structure, iff $\pi$ is so, and, in this case, $\pi^{-1} = 2\operatorname{Re} \Pi^{-1}$.

A \emph{complex Lie groupoid} is a Lie groupoid $\mathcal G \rightrightarrows M$ equipped with a multiplicative complex structure $j_{\mathcal G}$. In particular, $M$ is equipped with a complex structure $j$ and all the structure maps of $\mathcal G$ are holomorphic. Put $X = (M, j)$. When we want to emphasize that both $\mathcal G$ and $M$ are complex manifold we write $\mathcal G \rightrightarrows X$.

The Lie algebroid of a complex Lie groupoid is a \emph{holomorphic Lie algebroid}. A holomorphic Lie algebroid is a Lie algebroid $A \to M$ equipped with an \emph{infinitesimal multiplicative} (\emph{IM}) \emph{complex structure}, that is a complex structure $j_A : TA \to TA$ which is an automorphism of the tangent Lie algebroid $TA \to TA$ (see \cite{LSX2009} for more details). Similarly as in the complex Lie groupoid case, it follows that $M$ possesses a, necessarily unique, complex structure $j : TM \to TM$ such that $(A, j_A) \to X := (M, j)$ is a holomorphic vector bundle and 1) the anchor $(A, j_A) \to TX$ is a holomorphic map, 2) holomorphic sections are preserved by the Lie bracket, and 3) the Lie bracket is complex bi-linear when restricted to holomorphic sections. Every holomorphic vector bundle $(A, j_A) \to X := (M, j)$ equipped with a Lie algebroid structure on $A \to M$ such that the above properties 1)--3) are fulfilled, arises in this way \cite{LSX2008}. Let $(A \to M, j_A)$ be a holomorphic Lie algebroid. The underlying Lie algebroid (obtained from $(A \to M, j_A)$ forgetting about the complex structure) is called the real Lie algebroid of $(A \to M, j_A)$ and will be denoted by $\operatorname{Re}A$. A holomorphic Lie algebroid $A$ is \emph{integrable} iff it is the Lie algebroid of a holomorphic Lie groupoid. So a holomorphic Lie algebroid $(A, j_A)$ is integrable iff $\operatorname{Re} A$ is so, and, in this case, the complex structure on the integrating groupoid $\mathcal G$ is the unique complex structure integrating the IM complex structure $j_A$. Notice that a holomorphic Poisson structure $\Pi$ on a complex manifold $X = (M, j)$ induces a holomorphic Lie algebroid structure, denoted $(T^\ast X)_\Pi$, on $T^\ast X$ in a (certain) canonical way: use the usual formulas to define the bracket and the anchor on holomorphic sections and extend to all sections by the Leibniz rule and $C^\infty (M)$-linearity. We say that $(X, \Pi)$ is integrable if $(T^\ast X)_\Pi$ is so. Additionally, we have $\operatorname{Re} (T^\ast X)_\Pi = (T^\ast M)_{4 \operatorname{Re} \Pi}$ \cite{LSX2008}.

A \emph{complex symplectic groupoid} is a complex Lie groupoid $(\mathcal G \rightrightarrows M, j_{\mathcal G})$ equipped with a multiplicative complex symplectic structure, i.e.~a multiplicative complex $2$-form $\Omega \in \Omega^{2,0} (\mathcal G)$ such that $\overline{\partial} \Omega = 0$ and $d \Omega = 0$ (hence $\partial \Omega = 0$ as well). We recall that a complex valued form is multiplicative if so are both its real and its imaginary part. It easily follows that a holomorphic symplectic groupoid $(\mathcal G \rightrightarrows M, j_{\mathcal G}, \Omega)$ is equivalent to a symplectic-Nijenhuis groupoid $(\mathcal G \rightrightarrows M, \omega, N)$ such that the multiplicative Nijenhuis tensor $N$ is a complex structure. Under this equivalence $(j_{\mathcal G}, \Omega)$ corresponds to $(\operatorname{Re} \Omega, j_{\mathcal G})$. Under the inverse equivalence $(\omega, N)$ corresponds to $(N, \omega - i \omega_N)$. Putting everything together we get the following

\begin{proposition}[Laurent-Gengoux, Sti\'enon, Xu {\cite{LSX2009}}]\label{prop:holP}
\quad
\begin{enumerate}
\item Let $(\mathcal G \rightrightarrows M, j_{\mathcal G}, \Omega)$ be a complex symplectic groupoid. Then $M$ is equipped with a unique complex structure $j$ and a unique holomorphic Poisson structure $\Pi$ such that the target is a holomorphic Poisson map. Holomorphic Poisson manifold $(M, j, \Pi)$ is called the differentiation of $(\mathcal G, j_{\mathcal G}, \Omega)$.
\item Conversely, let $(M, j, \Pi)$ be a holomorphic Poisson manifold. Then $(M, j, \Pi)$ is integrable iff $(M, 4 \operatorname{Re} \Pi)$ (or equivalently $(M, \operatorname{Re} \Pi)$) is so. In this case, let $\mathcal G$ be the source-simply connected groupoid of the cotangent algebroid $(T^\ast M)_{4 \operatorname{Re} \Pi}$. Then $\mathcal G$ is equipped with a unique multiplicative complex structure $j_{\mathcal G}$ and a unique multiplicative complex symplectic form $\Omega$ inducing precisely the complex structure $j$ and the holomorphic Poisson structure $\Pi$ on $M$ by differentiation. The complex symplectic groupoid $(\mathcal G, j_{\mathcal G})$ is called the integration of $(M, j, \Pi)$. Finally, the real symplectic groupoids $(\mathcal G, \operatorname{Re} \Omega)$ integrates the real Poisson manifold $(M, 4 \operatorname{Re} \Pi)$.
\end{enumerate}
Differentiation and Integration establish a one-to-one correspondence between integrable holomorphic Poisson manifolds and complex symplectic groupoids. Shortly, holomorphic Poisson manifolds integrate to complex symplectic groupoids.
\end{proposition}

    \subsection{Homogeneous Poisson}
    
   As we mean to use the homogenization scheme to discuss integration of holomorphic Jacobi manifolds, we first discuss the integration problem for homogeneous Poisson manifolds. Recall that a \emph{homogeneous Poisson manifold} is a manifold $M$ equipped with a \emph{homogeneous Poisson structure} i.e.~a pair $(\pi, \zeta)$ consisting of a Poisson bivector $\pi$ and a vector field $\zeta$ such that $\mathcal L_\zeta \pi = [\zeta, \pi]^S = -\pi$. Vector field $\zeta$ is called the \emph{homogeneity vector field}.
   
\begin{example}
The homogenization $(\widetilde M, \widetilde J, Z)$ of a Jacobi manifold $(M, L, J)$ is a homogeneous Poisson manifold. Here $Z$ is the Euler vector field on $\widetilde M$.
\end{example}   
   
    If $\pi$ is non-degenerate and $\omega$ is the corresponding symplectic form, then $\mathcal L_\zeta \omega = d i_\zeta \omega = \omega $ and we say that $(M, \omega, \zeta)$ is a \emph{homogeneous symplectic manifold}. A homogeneous Poisson manifold $(M, \pi, \zeta)$ is integrable if the underlying Poisson structure $\pi$ is so.
    
\begin{remark}
For later use we collect in this remark some useful facts about multiplicative vector fields on Lie groupoids and their infinitesimal counterparts: \emph{Lie algebroid derivations}. For more details see \cite{MX1998}. Let $A \Rightarrow M$ be a Lie algebroid. A \emph{Lie algebroid derivation} of $A$ is a derivation $\delta : \Gamma (A) \to \Gamma(A)$ which is also a derivation of the Lie bracket $[-,-]$ on $\Gamma (A)$, i.e.
\[
\delta [\alpha, \beta] = [\delta \alpha, \beta] + [\alpha, \delta \beta].
\]
for all $\alpha, \beta \in \Gamma (A)$. It then follows that $\delta$ does also preserve the anchor $\rho$ in the sense that
\[
[\sigma (\delta), \rho (\alpha)] = \rho (\delta \alpha).
\] 
Let $\mathcal G \rightrightarrows M$ be a Lie groupoid with Lie algebroid $A$ and let $Z$ be a multiplicative vector field on it. Then $A$ is equipped with a canonical Lie algebroid derivation $\delta$, the differentiation of $Z$, given by formula
\begin{equation}\label{eq:mult_vf}
\begin{aligned}
\overrightarrow{\delta \alpha} = [Z, \overrightarrow{\alpha}]
\end{aligned}
\end{equation}
 Conversely if $A \Rightarrow M$ is an integrable Lie algebroid equipped with a Lie algebroid derivation and $\mathcal G$ is its source-simply connected integration, then $\mathcal G$ is equipped with a unique multiplicative vector field $Z$, the integration of $\zeta$, such that (\ref{eq:mult_vf}) holds. Integration and differentiation establish a one-to-one correspondence between integrable Lie algebroids equipped with a Lie algebroid derivation $\delta$ and source-simply connected Lie groupoids equipped with a multiplicative vector field $Z$. Additionally, $Z$ and $\delta$ are related by the following formula
for all $\alpha \in \Gamma (A)$. It follows that the symbol of $\delta$ agrees with the restriction $Z|_M$.
\end{remark}

\begin{lemma}\label{lem:hP}
Let $(M, \pi)$ be a Poisson manifold, and let $\zeta$ be a vector field on $M$. Then $(\pi, \zeta)$ is a homogeneous Poisson structure iff $\mathcal L_\zeta - \mathbb 1 : \Omega^1 (M) \to \Omega^1(M)$ is a derivation of the cotangent algebroid $(T^\ast M)_\pi$.
\end{lemma} 

\begin{proof}
It easily follows from the (Leibniz-type) formula
\begin{equation}\label{eq:der_bracket}
\mathcal L_\zeta [ \alpha, \beta]_\pi = [\mathcal L_\zeta \alpha, \beta]_\pi + [\alpha, \mathcal L_\zeta \beta]_\pi + [\alpha, \beta]_{\mathcal L_\zeta \pi},
\end{equation}
for all $\alpha, \beta \in \Omega^1 (M)$. Here, given any bivector $\pi$ (non necessarily Poisson) $[-,-]_{\pi}$ denotes the biderivation of $T^\ast M$ given by the same formula (\ref{eq:brack_pi}) as for Poisson bivectors. Equation (\ref{eq:der_bracket}) can be proved with a straightforward computation.
\end{proof}

\begin{definition}
A \emph{homogeneous symplectic groupoid} is a symplectic groupoid $(\mathcal G, \Omega)$ equipped with a multiplicative vector field $Z$ such that $(\Omega, Z)$ is a homogeneous symplectic structure, i.e.~$\mathcal L_Z \Omega = \Omega$.
\end{definition}

\begin{proposition}\label{prop:hP}
\quad
\begin{enumerate}
\item Let $(\mathcal G \rightrightarrows M, \Omega, Z)$ be a homogeneous 
symplectic groupoid. Then $M$ is equipped with a unique homogeneous Poisson 
structure $(\pi, \zeta)$ such that $(M, \pi)$ differentiates $(\mathcal G, \Omega)$ 
and $\zeta = Z|_M$. The homogeneous Poisson manifold $(M, \pi, \zeta)$ is called
 the differentiation of $(\mathcal G, \Omega, Z)$.
\item Conversely, let $(M, \pi, \zeta)$ be an \emph{integrable} homogeneous Poisson manifold, and let $(\mathcal G, \Omega)$ be the (source-simply connected) integration of $(M, \pi)$. Then $\mathcal G$ is equipped with a unique multiplicative homogeneity vector field $Z$ for $\Omega$, such that $(M, \pi, Z)$ differentiates $(\mathcal G, \Omega, Z)$. The homogeneous symplectic groupoid $(\mathcal G, \Omega, Z)$ is called the integration of $(M, \pi, \zeta)$.
\end{enumerate}
Integration and differentiation establish a one-to-one correspondence between integrable homogeneous Poisson manifolds and source-simply connected homogeneous symplectic groupoids. Shortly, homogeneous Poisson manifolds integrate to homogeneous symplectic groupoids.
\end{proposition}

\begin{proof}
Begin with a symplectic groupoid $(\mathcal G \rightrightarrows M, \Omega)$ and a multiplicative vector field $Z$ on it. Then 1) $M$ possesses a Poisson structure $\pi$, 2) $Z$ induce a derivation $\delta$ of the cotangent algebroid $(T^\ast M)_\pi$, 3) the symbol of $\delta$ is $\zeta = Z|_M$. We will show that $(\Omega, Z)$ is a homogeneous symplectic structure iff $\delta = \mathcal L_{\zeta} - \mathbb 1$. The theorem will then follow immediately. So, notice that $\mathcal L_Z \Omega$ is multiplicative and compute its Spencer operator $(\mathcal D, \ell)$. As $\mathcal L_Z \Omega$ is closed, it is enough to take care of the second entry only. So let $\alpha \in \Omega^1 (M)$ and compute
\begin{equation}\label{eq:3}
\begin{aligned}
\ell(\alpha) & = 1^\ast (i_{\overrightarrow \alpha}\mathcal L_Z \Omega) = 1^\ast (i_{[\overrightarrow \alpha, Z]} \Omega) + \mathcal L_\zeta 1^\ast (i_{\overrightarrow \alpha}\Omega) \\ & = - 1^\ast (i_{\overrightarrow{\delta\alpha}}\Omega) + \mathcal L_\zeta \alpha = \mathcal L_\zeta \alpha - \delta \alpha
\end{aligned}
\end{equation}
which agrees with $\mathbb 1 \alpha = \alpha$ for all $\alpha$ iff $\delta = \mathcal L_{\zeta} - \mathbb 1$ as claimed.
\end{proof}

    \subsection{Homogeneous Poisson-Nijenhuis}

\begin{definition}
A \emph{homogeneous Poisson-Nijenhuis manifold} is a manifold equipped with a \emph{homogeneous Poisson-Nijenhuis structure}, i.e.~a triple $(\pi, N, \zeta)$ such that, $(\pi, N)$ is a Poisson-Nijenhuis structure, $(\pi, \zeta)$ is a homogeneous Poisson structure, and, additionally, $\mathcal L_\zeta N = 0$. If $\pi$ is non-degenerate we speak about \emph{homogeneous symplectic-Nijenhuis manifolds} and \emph{homogeneous symplectic-Nijenhuis structures}. A \emph{homogeneous symplectic-Nijenhuis groupoid} is a groupoid equipped with a \emph{multiplicative homogeneous symplectic-Nijenhuis structure}, i.e.~a homogeneous symplectic-Nijenhuis structure such that all components are multiplicative.
\end{definition}

The following proposition refines Propositions \ref{prop:PN} and \ref{prop:hP} for homogeneous Poisson-Nijenhuis manifolds.

\begin{proposition}\label{prop:hPN}
\quad
\begin{enumerate}
\item Let $(\mathcal G \rightrightarrows M, \Omega, \mathcal N, Z)$ be a homogeneous symplectic groupoid. Then $M$ is equipped with a unique homogeneous symplectic-Nijenhuis structure $(\pi, N, \zeta)$ such that $(M, \pi)$ differentiates $\mathcal G$, and additionally $N = \mathcal N$ and $\zeta = Z|_M$. The homogeneous symplectic-Nijenhuis manifold $(M, \pi, N, \zeta)$ is called the differentiation of $(\mathcal G, \Omega, \mathcal N, Z)$.
\item Conversely, let $(M, \pi, N, \zeta)$ be an \emph{integrable} homogeneous Poisson-Nijenhuis manifold, i.e.~$(M, \pi, N)$ is integrable, and let $(\mathcal G, \Omega, \mathcal N)$ be its integration. Then $\mathcal G$ is equipped with a  unique  multiplicative homogeneity vector field $Z$ for $(\Omega, \mathcal N)$ such that $(M, \pi, N, \zeta)$ differentiates $(\mathcal G, \Omega, \mathcal N, Z)$. The homogeneous Poisson-Nijenhuis groupoid $(\mathcal G, \Omega, \mathcal N, Z)$ is called the integration of $(M, \pi, N, \zeta)$.
\end{enumerate}
Differentiation and Integration establish a one-to-one correspondence between integrable homogeneous Poisson-Nijenhuis manifolds and source-simply connected homogeneous symplectic-Nijenhuis groupoids. Shortly, homogeneous Poisson-Nijenhuis manifolds integrate to homogeneous symplectic-Nijenhuis groupoids.
\end{proposition}

\begin{proof}
We already now that Poisson-Nijenhuis manifolds integrate to symplectic-Nijenhuis groupoids, and that homogeneous Poisson manifolds integrate to homogeneous symplectic groupoids. So begin with a symplectic-Nijenhuis groupoid $(\mathcal G \rightrightarrows M, \Omega)$ equipped with 1) a multiplicative symplectic-Nijenhuis structure $(\Omega, N)$, and 2) a homogeneous symplectic structure $(\Omega, Z)$. Denote by $(\pi, N)$ and by $(\pi, \zeta)$ the induced Poisson-Nijenhuis and homogeneous Poisson structures on $M$. It remains to check that $\mathcal L_\zeta N = 0$ iff $\mathcal L_Z \mathcal N = 0$. It is easy to check that  $\mathcal L_Z \mathcal N = 0$ is equivalent to $\mathcal L_Z \Omega_{\mathcal N} = \Omega_{\mathcal N}$. So compute the Spencer operator $(\mathcal D, \ell)$ of $\mathcal L_Z \Omega_{\mathcal N}$. As $\mathcal L_Z \Omega_{\mathcal N}$ is closed, it is enough to take care of the second entry only. Recall that the Spencer operator of $\Omega_{\mathcal N}$ is $(d \circ N^\ast, N^\ast)$, and the Lie algebroid derivation differentiating $Z$ is $\mathcal L_\zeta - \mathbb 1$. Then, a similar computation as (\ref{eq:3}) shows that
\[
\ell (\alpha) = 1^\ast (i_{\overrightarrow \alpha} \mathcal L_Z \Omega_{\mathcal N}) = - N^\ast (\mathcal L_\zeta \alpha - \alpha) + \mathcal L_\zeta N^\ast \alpha = N^\ast \alpha + (\mathcal L_\zeta N)^\ast \alpha
\]
 for all $\alpha \in \Omega^1 (M)$. So $\ell = N^\ast$ iff $\mathcal L_\zeta N = 0$ as claimed.
\end{proof}

\subsection{Homogeneous holomorphic Poisson}

Recall that a complex vector field on a Lie groupoid is multiplicative if so are its real and its imaginary part. 

\begin{definition}
A \emph{homogeneous holomorphic Poisson manifold} is a holomorphic Poisson manifold $(X, \Pi)$ equipped with a \emph{holomorphic vector field $Z$} such that $[Z, \Pi] = - \Pi$. If $\Pi$ is non-degenerate, we speak about a \emph{homogeneous complex symplectic manifold}. A \emph{homogeneous complex symplectic groupoid} is a complex symplectic groupoid $(\mathcal G, \Omega)$ equipped with a \emph{multiplicative holomorphic vector field} $\mathcal Z$ such that $\mathcal L_{\mathcal Z} \Omega = \Omega$.
\end{definition}

\begin{remark}
A homogeneous holomorphic Poisson manifold $(M, j, \Pi,  Z)$ is equivalent to a homogeneous Poisson-Nijenhuis manifold $(M, \pi, N, \zeta)$ such that the Nijenhuis tensor $N$ is an almost complex (hence complex) structure. Under this equivalence $(j, \Pi, Z)$ corresponds to $(\operatorname{Re} \Pi, j, 2\operatorname{Re} Z)$. Under the inverse equivalence $(\pi, N, \zeta)$ corresponds to $(N, \pi - i \pi_N, (\zeta -i j \zeta)/2)$.
\end{remark}

The following proposition refines Proposition \ref{prop:holP} for homogeneous holomorphic Poisson manifolds.

\begin{proposition}\label{prop:hcsg}
\quad
\begin{enumerate}
\item Let $(\mathcal G \rightrightarrows M, j_{\mathcal G}, \Omega, \mathcal Z)$ be a homogeneous complex symplectic groupoid. Then $M$ is equipped with a unique complex structure $j$ and a unique homogeneous holomorphic Poisson structure $(\Pi, Z)$ such that $(M, j, \Pi)$ differentiates $\mathcal G, j_{\mathcal G}, \Pi)$ and $Z = \mathcal Z|_M$. The homogeneous holomorphic Poisson manifold $(M, j, \Pi, Z)$ is called the differentiation of $(\mathcal G, j_{\mathcal G}, \Omega, \mathcal Z)$.
\item Conversely, let $(M, j, \Pi, Z)$ be an \emph{integrable} homogeneous holomorphic Poisson manifold, i.e.~$(M, j, \Pi)$ is integrable, and let $(\mathcal G, j_{\mathcal G}, \Omega)$ be the (source-simply connected) integration of the latter. Then $\mathcal G$ is equipped with a unique holomorphic homogeneity vector field for $\Omega$, such that $(M, j, \Pi, Z)$ dfferentiates $(\mathcal G, j_{\mathcal G}, \Omega, \mathcal Z)$. The homogeneous complex symplectic groupoid $(\mathcal G, j_{\mathcal G}, \Omega, \mathcal Z)$ is called the integration of $(M, j, \Pi, Z)$. 
\end{enumerate}
Integration and differentiation establish a one-to-one correspondence between integrable homogeneous holomorphic Poisson manifolds and source-simply connected homogeneous complex symplectic groupoids. Shortly, homogeneous holomorphic Poisson manifolds integrate to homogeneous complex symplectic groupoids.
\end{proposition}

\begin{proof}
The proof easily follows from Propositions \ref{prop:holP} and \ref{prop:hPN}. 
\end{proof}
    
\section{Integration of Jacobi structures}\label{Sec:Cont_Group}
    \subsection{Integration of Jacobi manifolds II}
    Let $(M, L, \{-,-\} \equiv J)$ be a Jacobi manifold, and let $(\widetilde M, \pi, Z)$ be its Poissonization. Recall that this means that $\widetilde M = L^\ast \smallsetminus 0$, $Z$ is the Euler vector field on it, and $\pi = \widetilde J$ is the homogenization of the skew-symmetric Atiyah $(0,2)$-tensor $J$. In particular $(\pi, Z)$ is a homogeneous Poisson structure. Denote by $p : \widetilde M \to M$ the projection. It is a principal $\mathbb R^\times$-bundle. Finally, recall the pull-back diagrams
\begin{equation}\label{diag}
\begin{array}{c}
\xymatrix{ \mathbb R_{\widetilde M} \ar[r] \ar[d] &  \widetilde M  \ar[d]^-p \\
L   \ar[r] & M }
\end{array}, \quad
\begin{array}{c}
\xymatrix{  T^\ast {\widetilde M} \ar[r] \ar[d]_-{p_{T^\ast}} & \widetilde M \ar[d]^-p \\
  \mathfrak J^1 L \ar[r] & M }
\end{array}.
\end{equation}
from Section \ref{Sec:hom_Atiyah}.

\begin{lemma}\label{lem:Poiss}
Map $\mathbb R_{\widetilde M} \to L$ is a Jacobi map, and $T^\ast \widetilde M \to \mathfrak J^1 L$ is a map of Lie algebroids $(T^\ast \widetilde M)_\pi \to (\mathfrak J^1 L)_J$. More precisely, the jet algebroid $(\mathfrak J^1 L)_J$ acts canonically on the fibration $\widetilde M$, and $(T^\ast \widetilde M)_\pi$ is the corresponding action algebroid. In particular, diagram (\ref{diag}) is a principal $\mathbb R^\times$-bundle in the category of Lie algebroids, i.e.~$\mathbb R^\times$ acts on $(T \widetilde M)_\pi$ by Lie algebroid automorphisms.
\end{lemma}

\begin{proof}
First of all notice that a Poisson manifold $(\widetilde M, \pi)$ can be seen as a Jacobi manifold with Jacobi bundle $\mathbb R_{\widetilde M}$ and Jacobi bracket given by the Poisson bracket $\{-,-\}_\pi$ corresponding to $\pi$. The first claim now follows from identity
\begin{equation}\label{eq:hom_bracket}
\{ \widetilde \lambda, \widetilde \mu \}_{\widetilde J} = \widetilde{\{\lambda, \mu \}}
\end{equation}
for all $\lambda, \mu \in \Gamma (L)$. The latter is just a special case of (\ref{eq:T_tilde}). The fact that $p_{T^\ast}$ is a Lie algebroid map follows from (\ref{eq:hom_bracket}), surjectivity of $p_{T^\ast}$,  identity $\widetilde{\mathfrak j^1 \lambda} = d \widetilde \lambda$, for all $\lambda \in \Gamma (L)$, and the fact that the Lie brackets $[-,-]_\pi$, and $[-,-]_J$, on $\Omega^1 (\widetilde M)$ and $\Gamma (\mathfrak J^1 L)$ respectively, are completely determines by properties
\[
[df, dg]_\pi = d\{f,g\}_\pi, \quad [\mathfrak j^1 \lambda, \mathfrak j^1 \mu]_J = \mathfrak j^1 \{\lambda, \mu\},
\]
for $f,g \in C^\infty (\widetilde M)$, and $\lambda, \mu \in \Gamma (L)$. For the second part of the statement, we argue as follows. As $p_{T^\ast}$ is a regular vector bundle map, then $T^\ast \widetilde M$ is actually (isomorphic to) the action Lie algebroid corresponding to an action of $(\mathfrak J^1 L)_J$ on the fibration $\widetilde M$. The action map $\Gamma ((\mathfrak J^1 L)_L) \to \mathfrak X (\widetilde M)$ is given by
\[
\psi \mapsto \pi^\sharp \widetilde \psi = \widetilde J{}^\sharp \widetilde \psi = \widetilde{J^\sharp \psi}.
\]

\end{proof}

\begin{remark}
The action of $\mathbb R^\times$ on $T^\ast \widetilde M$ can be described as follows. We already remarked that every covector $\theta \in T^\ast \widetilde M$ over a point $\epsilon \in \widetilde M$ is of the form $d_\epsilon \widetilde \lambda$ for some $\lambda \in \Gamma (L)$. Then, for $r \in \mathbb R^\times$, we have $r . \theta = d_{r \cdot \epsilon} \widetilde \lambda$.
\end{remark}

A theorem equivalent to the following one has been first proved by Crainic and Zhu in the case when $L$ is a trivial line bundle (see \cite{CZ2007}).

\begin{theorem}\label{theor:Jac_hom_Poiss}
Jacobi manifold $(M, L, J)$ is integrable iff its Poissonization $(\widetilde M, \pi, Z)$ is so. In this case, let $(\mathcal G, H)$ be the source-simply connected contact groupoid integrating $(M, L, J)$, and let $(\widetilde{\mathcal G}, \omega, \mathcal Z)$ be the source-simply connected homogeneous symplectic groupoid integrating $(\widetilde M, \pi, Z)$. Then
\begin{enumerate}
\item $\mathcal G$ acts on the fibration $\widetilde M \to M$,
\item $\widetilde{\mathcal G}$ is the corresponding action groupoid, and
\item $(\widetilde{\mathcal G}, \omega, \mathcal Z)$ is the homogenization of $(\mathcal G, H)$.
\end{enumerate}
\end{theorem}

\begin{proof}
Suppose $(M, L, J)$ is integrable. This means that the jet algebroid $(\mathfrak J^1 L)_J$ is integrable. Hence, every action algebroid is integrated by the corresponding action groupoid. In particular, from Lemma \ref{lem:Poiss}, the cotangent algebroid $(T^\ast \widetilde M)_\pi$ is integrable. Conversely let $(\widetilde M, \pi, Z)$ be integrable. This means that the cotangent algebroid $(T^\ast \widetilde M)_\pi$ is integrable. Let $\widetilde{\mathcal G} \rightrightarrows \widetilde M$ be its source-simply connected integration. The $\mathbb R^\times$-action from Lemma \ref{lem:Poiss} does now integrate to a (necessarily free and proper) $\mathbb R^\times$-action on $\widetilde{\mathcal G}$ by groupoid automorphisms. It follows from Lemma \ref{lem:Poiss} again that $\widetilde{\mathcal G}/\mathbb R^\times \rightrightarrows M$ is a, necessarily source-simply connected, groupoid integrating $(\mathfrak J^1 L)_J$.

Now, assume that one, hence both, of $(\mathfrak J^1 L)_J$ and $(T^\ast \widetilde M)_\pi$ is integrable, and let $(\mathcal G, H)$ and $(\widetilde{\mathcal G}, \omega, \mathcal Z)$ be as in the statement. Then, as we have seen above, we have the following pull-back diagram, via the target:
\[
\xymatrix{ \widetilde{\mathcal G} \ar@<-2pt>[r] \ar@<+2pt>[r] \ar[d] & \widetilde M \ar[d]^-p\\
                \mathcal G \ar@<-2pt>[r] \ar@<+2pt>[r] & M}
\] 
meaning that $\widetilde{\mathcal G}$ is (canonically isomorphic to) the action groupoid induced by an action of $\mathcal G$ on $\widetilde M$. From Lemma \ref{lem:Poiss} this action integrates the canonical action of $(\mathfrak J^1 L)_J$. This proves (1) and (2). For (3) recall that $\omega$ is completely determined by the condition that the source of $\widetilde{\mathcal G}$ is a Poisson map, or, equivalently, a Jacobi map. Similarly, $H$ is completely determined by the condition that the source of $\widetilde{\mathcal G}$ is a Jacobi map \cite{CS2015}% (for the last remark, recall that the normal bundle $T \mathcal G /H$ is canonically isomorphic to $t^\ast L$, where $t : \mathcal G \to M$ is the target)
. Now, the vector field $\mathcal Z$ on $\widetilde{\mathcal G}$ is, by construction, the Euler vector field for a principal $\mathbb R^\times$-bundle structure $\widetilde{\mathcal G}$. Hence $(\omega, \mathcal Z)$ induce a unique contact structure $\widetilde H$ on $\mathcal G$ such that $\omega$ is the homogenization of the symplectic Atiyah $2$-form of $\widetilde H$. As all the projections $s: (\widetilde{\mathcal G}, \omega) \to (\widetilde M, \pi)$, $p : (\widetilde M, \pi) \to (M, L, J)$ and $(\widetilde{\mathcal G}, \omega) \to (\mathcal G, \widetilde H)$ are Jacobi maps, it follows that $s: (\mathcal G, \widetilde H) \to (M, L, J)$ is also a Jacobi map. So $\widetilde H = H$ provided only $\widetilde H$ is multiplicative, which follows from the general discussion in Example \ref{ex:mult_contact}.
\end{proof}

    \subsection{Jacobi-Nijenhuis}\label{Sec:JN}
   Recall from Example \ref{ex:Nijenhuis} (see also \cite{MMP1999}), that a \emph{Jacobi-Nijenhuis manifold} is a manifold $M$ equipped with a \emph{Jacobi-Nijenhuis structure}, i.e.~a triple $(L, J, N)$ consisting of a line bundle $L \to M$, a skew-symmetric biderivation $J$ of $L$, and an Atiyah $(1,1)$-tensor $N : DL \to DL$ such that
\begin{enumerate}
\item $(L,J)$ is a Jacobi structure,
\item $N$ is Nijenhuis (i.e.~its torsion vanishes)
\item the following formulas (which can be actually written in several equivalent ways) hold:
\begin{align}
J (\psi, N^\dag \chi) + J (\chi, N^\dag \psi) & = 0  \label{eq:1Jac}\\
\mathcal L_{J^\sharp \psi} N^\dag \chi - \mathcal L_{J^\sharp \chi} N^\dag \psi - d_D J (N^\dag \psi, \chi) & = N^\dag [\psi, \chi]_J. \label{eq:2Jac}
\end{align}
\end{enumerate}

Here $N^\dag : \mathfrak J^1 L \to \mathfrak J^1 L$ is the transpose of $N$ (twisted by $L$). From (\ref{eq:1Jac}), biderivation $J_N = J (N -,-)$ is actually skew-symmetric. It then follows from (\ref{eq:2Jac}) and the other axioms, that $J_N$ is Jacobi. If $J$ is non degenerate, then it comes from a contact structure $H$ with the property that $TM/H = L$  and we call $(H, N)$ a \emph{contact-Nijenhuis} structure.

We summarize the remarks in Example \ref{ex:Nijenhuis} in the following

\begin{lemma}
Let $(M, L, J, N)$ be a Jacobi-Nijenhuis manifold, and let $(\widetilde M, \pi, Z)$ be the Poissonization of the underlying Jacobi manifold, i.e.~$\widetilde M = L^\ast \smallsetminus 0$, $\pi = \widetilde J$, and $Z$ is the Euler vector field on $\widetilde M$. Additionally, let $\widetilde N$ be the homogenization of $N$. Then $(\widetilde M, \pi, \widetilde N, Z)$ is a homogeneous Poisson-Nijenhuis manifold.  
\end{lemma}

\begin{definition}
A \emph{contact-Nijenhuis groupoid} is a Lie groupoid $\mathcal G \rightrightarrows M$ equipped with a \emph{multiplicative contact-Nijenhuis structure}, i.e.~a contact-Nijenhuis structure $(H, \mathcal N)$ such that both $H$ and $\mathcal N$ are multiplicative. 
\end{definition}

%\begin{remark}
%Let $\mathcal G \rightrightarrows M$ be a Lie groupoid, and let $(L_{\mathcal G}, G; 0_M, M)$ be an LB-groupoid over it. It is easy to see that an Atiyah $(1,1)$-tensor $\mathcal N$ on $L_{\mathcal G}$ is multiplicative iff, seen as a map $\mathcal N : DL_{\mathcal G} \to DL_{\mathcal G}$, it is a map of VB-groupoids.
%\end{remark}

\begin{proposition}\label{prop:JN}
\quad
\begin{enumerate}
\item Let $(\mathcal G \rightrightarrows M, H, \mathcal N)$ be a contact-Nijenhuis groupoid. Then $M$ is equipped with a unique Jacobi-Nijenhuis structure $(L, J, N)$ such that $(M, L, J)$ differentiates $(\mathcal G, H)$ and $N = \mathcal N|_{DL}$. The Jacobi-Nijenhuis manifold $(M, L, J, N)$ is called the differentiation of $(\mathcal G, H, \mathcal N)$.
\item Conversely, let $(M, L, J, N)$ be an \emph{integrable} Jacobi-Nijenhuis manifold, i.e.~the underlying Jacobi manifold $(M, L, J)$ is integrable, and let $(\mathcal G, H)$ be the (source-simply) integration of the latter. Then $\mathcal G$ is equipped with a unique multplicative Nijenhuis tensor $\mathcal N$ such that $(\mathcal G, H, \mathcal N)$ is a contact-Nijenhuis groupoid and $(M, L, J, N)$ is its differentiation. $(\mathcal G, H, \mathcal N)$ is called the integration of $(M, L, J, N)$.
\end{enumerate}
Integration and differentiation establish a one-to-one correspondence between integrable Jacobi-Nijenhuis manifolds and source-simply connected contact-Nijenhuis groupoids. Shortly, Jacobi-Nijenhuis manifolds integrate to contact-Nijenhuis groupoids. 
\end{proposition}

\begin{proof}
Let $(\mathcal G, H, \mathcal N)$ be as in the statement. Additionally, let $(\widetilde{\mathcal G}, \Omega, \widetilde{\mathcal N}, \mathcal Z)$ be the homogenization of $(\mathcal G, H, \mathcal N)$. We already know that $\Omega$ and $Z$ are multiplicative. From Proposition \ref{prop:mult_Atiyah}, multiplicativity of $\widetilde{\mathcal N}$ is equivalent to that of $\mathcal N$. We conclude that $(\widetilde{\mathcal G}, \Omega, \widetilde{\mathcal N}, \mathcal Z)$ is a homogeneous symplectic-Nijenhuis groupoid. So it induces a homogeneous Poisson-Nijenhuis structure $(\pi, \widetilde N, Z)$ on $\widetilde M$ by differentiation. Additionally, $Z$ is the Euler vector field. Hence $M$ is canonically equipped with a Jacobi-Nijenhuis structure (with the core $L$ of the LB-groupoid $(\nu (H) \rightrightarrows 0_M; \mathcal G \rightrightarrows M)$ as its Jacobi bundle) with the desired properties.

Conversely, let $(M, L, J, N)$ be an integrable Jacobi-Nijenhuis manifold, and let $(\mathcal G, H)$ be the source-simply connected integration of $(M, L, J)$, so that, in particular, $L_\mathcal G := T\mathcal G / H = t^\ast L$.  The homogenization $(\widetilde M, \pi, \widetilde N, Z)$ of $(M, L, J, N)$ is also integrable. Specifically $(\widetilde M, \pi, Z)$ integrates to the symplectization $(\widetilde{\mathcal G}, \Omega, \mathcal Z)$ of $(\mathcal G, H)$. Additionally, there is a unique multiplicative Nijenhuis tensor $\widetilde{\mathcal N}$ on $\widetilde{\mathcal G}$ such that $(\widetilde{\mathcal G}, \Omega, \widetilde{\mathcal N}, \mathcal Z)$ is a homogeneous symplectic-Nijenhuis groupoid and $\widetilde{\mathcal N}|_{T \widetilde M} = \widetilde N$. In particular, $\widetilde{\mathcal N}$ is the homogenization of a unique Atiyah $(1,1)$-tensor $\mathcal N : DL_{\mathcal G} \to DL_{\mathcal G}$ such that $(H, \mathcal N)$ is a contact-Nijenhuis structure. From Proposition \ref{prop:mult_Atiyah} $\mathcal N$ is multiplicative, so that $(\mathcal G, H, \mathcal N)$ is a contact-Nijenhuis groupoid with the desired properties.
\end{proof}

    \subsection{Holomorphic contact groupoids}
    
    We finally come to the holomorphic contact groupoids of the title. We refer to \cite{VW2016b} for conventions about holomorphic contact structures.
    
    \begin{definition}\label{def:compl_cont_grpd}
A \emph{holomorphic contact groupoid} is a complex groupoid $\mathcal G \rightrightarrows X = (M, j)$ equipped with a \emph{multiplicative holomorphic contact structure} $H$, i.e.~a multiplicative holomorphic contact distribution $H \subset T^{1,0} \mathcal G$ which covers $T^{1,0}X$, i.e.~$ds (H) = T^{1,0}X$.
\end{definition}

\begin{remark}
Definition \ref{def:compl_cont_grpd} needs some explanations. First of all, precisely as in the linear case, the $(1,0)$-tangent bundle $T^{1,0} \mathcal G$ (with the obvious structure maps) is a groupoid over $T^{1,0} X$. Then, a holomorphic distribution $\mathcal D \subset T^{1,0} \mathcal G$ is \emph{multiplicative} if it is a subgroupoid in the groupoid $T^{1,0} \mathcal G \rightrightarrows T^{1,0} X$ over a possibly smaller space of objects $\mathcal D_0 \subset T^{1,0} X$: $\mathcal D \rightrightarrows \mathcal D_0$. For a contact distribution $H$ we additionally require that $H_0 = T^{1,0} X$. In this case, one can prove, precisely as in the real case, that the normal bundle $L_{\mathcal G} := T^{1,0} \mathcal G / H$ sits in a (complex) VB-groupoid $(L_{\mathcal G}\rightrightarrows 0_X; \mathcal G\rightrightarrows X)$ with trivial side bundle. In particular, $\mathcal G$ acts on the core $L$ of $(L_{\mathcal G}\rightrightarrows 0_X; \mathcal G\rightrightarrows X)$ and $L_{\mathcal G} \simeq t^\ast L$. Notice that both $L_{\mathcal G}$ and $L$ are automatically holomorphic line bundles. 
\end{remark}

holomorphic contact structures are equivalent to complex symplectic Atiyah $2$-forms exactly as in the real case. Namely, a holomorphic contact structure $H$ on a complex manifold $X = (M,j)$ determines the holomorphic $L_{\mathcal G}$-valued $1$-form
\[
\theta_H : T^{1,0} \mathcal G \to L := T^{1,0}X/H, \quad \xi \mapsto \xi + H.
\]
Composing with the symbol we get the holomorphic Atiyah $1$-form $\Theta := \theta_H \circ \sigma : D^{1,0} L \to L$. Recall that $D^{1,0} L$ is a holomorphic Lie algebroid. Denote by $\partial_D$ the associated differential on holomorphic forms. Then, exactly as in the real case, $\omega = \partial_D \Theta$ is a non-degenerate, holomorphic Atiyah $2$-form, what we call a \emph{complex symplectic Atiyah $2$-form}. Homogenizing $\omega$ we get a homogeneous complex symplectic manifold $(\widetilde X, \widetilde \omega, Z)$, the \emph{symplectization} of $(\mathcal G, H)$, which contains a full information on $(X, H)$. Here $\widetilde X = L^\ast \smallsetminus 0$, where $L^\ast$ is the complex dual to $L$, and $Z$ is the (holomorphic) Euler vector field on $\widetilde X$. We leave the obvious details to the reader. 

Now let $(\mathcal G, H)$ be a holomorphic contact groupoid. Its symplectization $(\widetilde {\mathcal G}, \Omega, \mathcal Z)$ is a homogeneous complex symplectic groupoid. In particular multiplicativity of $H$ is equivalent to multiplicativity of $\Omega$ exactly as in the real case. In the complex case, however, homogenization can be performed in $2$-steps (see Section \ref{Sec:complex_case}). First let $\widehat{\mathcal G}$ be the real projective bundle of the complex dual $L_\mathcal G^\ast \to \mathcal G$ of $L_{\mathcal G}$. The proof of the following lemma is straightforward and it is left to the reader.

\begin{lemma}
There is a canonical Lie groupoid structure $\widehat{\mathcal G} \rightrightarrows \widehat M$, where $\widehat M$ is the real projective bundle of the complex dual $L^\ast \to M$ of $L$. More precisely, diagram
\[
\xymatrix{ \widehat{\mathcal G} \ar@<-2pt>[r] \ar@<+2pt>[r] \ar[d]_-{q_{\mathcal G}} & \widehat M \ar[d]^-q\\
                \mathcal G \ar@<-2pt>[r] \ar@<+2pt>[r] & M}
\]  
is a principal $U(1)$-bundle in the category of real Lie groupoids, meaning that $U(1)$ acts by Lie groupoid automorphisms. As it is also a pull-back diagram, via the target, $\widehat{\mathcal G} \rightrightarrows \widehat M$ is the action groupoid corresponding to the obvious action of $\mathcal G$ on $\widehat M$ (the one induced by the action on $L$).
\end{lemma}

Now, recall from \cite{VW2016b}, that $H$ induces a real contact structure $\widehat H$ on $\widehat{\mathcal G}$ (see also \cite{K1959}). To see this, first symplectize $(\mathcal G, H)$ to $(\widetilde{\mathcal G}, \Omega, \mathcal Z)$, and use that 1) $(\operatorname{Re}\Omega, 2 \operatorname{Re} \mathcal Z)$ is a real homogeneous symplectic structure, and 2) $\widetilde{\mathcal G}$ is a principal $\mathbb R^\times$-bundle over $\widehat{\mathcal G}$ with Euler vector field given by $2 \operatorname{Re} \mathcal Z$ (for more details see \cite{VW2016b}), so that 3) $\operatorname{Re}\Omega$ is the symplectization of a unique contact structure $\widehat H$ on $\widehat{\mathcal G}$.

\begin{lemma}\label{lem:mult_H}
The multiplicativity of $\widehat H$ is equivalent to that of $H$.
\end{lemma}

\begin{proof}
There is a sequence of equivalences: $H$ is multiplicative iff $\Omega$ is so, iff $\operatorname{Re} \Omega$ is so, iff $\widehat H$ is so. The first one, as we already mentioned, can be proved exactly as in the real case. The second one is straightforward, just use the definition of multiplicative form in one direction, and use multiplicativity of the complex structure on $\widetilde{\mathcal G}$ in the other direction. Last equivalence follows from the fact that $\operatorname{Re} \Omega$ is the symplectization of $\widehat H$ and the discussion in Example \ref{ex:contact}.
\end{proof}

As in Section \ref{Sec:hom_trick}, denote by $q : \widehat M \to M$ the projection, and let $\widehat L \to \widehat M$ be the dual of the tautological bundle, so that $\widehat L = q^\ast L$. Similarly, denote by $q_{\mathcal G}: \widehat{\mathcal G} \to \mathcal G$ the projection, and let $\widehat L_{\mathcal G} \to \widehat{\mathcal G}$ be the dual of the tautological bundle on $\widehat{\mathcal G}$, so that $\widehat L_{\mathcal G} = q_{\mathcal G}^\ast L_{\mathcal G} = t^\ast \widehat L$. More precisely, $\widehat{\mathcal G}$ acts on $\widehat L$, and $(\widehat L_{\mathcal G}\rightrightarrows 0_{\widehat M}; \widehat{\mathcal G}\rightrightarrows \widehat M)$ is the LB-groupoid corresponding to this action. Finally, recall that the holomorphic line bundle structure on $L \to X$ induces a (torsion free) complex structure on the gauge algebroid $D \widehat L \to \widehat M$. By the same reason, the holomorphic line bundle structure on $L_{\mathcal G} \to \mathcal G$ induces a (torsion free) complex structure $\widehat j_{\mathcal G} : D \widehat L_{\mathcal G} \to D \widehat L_{\mathcal G}$. 

\begin{proposition}
Let $(\mathcal G, H)$ be a holomorphic contact groupoid, and let $\widehat{\mathcal G}$ be as above. Then $\widehat j_{\mathcal G} : D\widehat L_{\mathcal G} \to D\widehat L_{\mathcal G}$ is a multiplicative Atiyah $(1,1)$-tensor, so that $(\widehat{\mathcal G}, \widehat H, \widehat j_{\mathcal G})$ is a contact-Nijenhuis groupoid.
\end{proposition}

\begin{proof}
The proof is analogous to the proof of Lemma \ref{lem:mult_H}. Namely, $j_{\widetilde{\mathcal G}}$ is the homogenization of $\widehat j_{\mathcal G}$. Hence, from Proposition \ref{prop:mult_Atiyah}, multiplicativity of the latter is equivalent to multiplicativity of the former.
 \end{proof}

    \subsection{Holomorphic Jacobi}\label{Sec:Hol_Jac}

Let $X = (M, j)$ be a complex manifold, and let $(L, J)$ be a holomorphic Jacobi structure on it. This means that $L \to X$ is a holomorphic line bundle, and 
\[
J : \wedge^2 \mathfrak J^1 L \to L
\]
is a holomorphic bi-derivation of $L$ such that $[J,J]^{SJ} = 0$. Here $\mathfrak J^1 L$ is the \emph{holomorphic jet bundle}. For more details on holomorphic Jacobi structures, in particular for details about their Poissonization, see \cite{VW2016b}. As we did above, denote by $\widehat M$ the projective bundle of the complex dual of $L$, and let $\widehat L \to \widehat M$ be the dual of the tautological bundle. Denote by $q : \widehat M \to M$ the projection. It is a principal $U(1)$-bundle projection. Recall that 
\begin{enumerate}
\item there is a (torsion free) complex structure $\widehat j : D \widehat L \to D \widehat L$;
\item there is a Jacobi structure $(\widehat L, \widehat J)$;
\item triple $(\widehat L, \widehat J, \widehat j)$ is a Jacobi-Nijenhuis structure.
\end{enumerate}
Additionally, the real Lie algebroid $(\mathfrak J^1 L)_{J, \mathrm{Re}}$ of the holomorphic jet algebroid $(\mathfrak J^1 L)_{J}$ of $(X, L, J)$ acts on $\widehat M$, and the jet algebroid $(\mathfrak J^1_{\mathbb R} \widehat L)_{4\widehat J}$ is canonically isomorphic to the corresponding action algebroid (here we use $\mathfrak J^1_{\mathbb R}$ for \emph{real} jets). In particular there is a projection of Lie algebroids
\[
(\mathfrak J^1_{\mathbb R} \widehat L)_{4\widehat J} \to (\mathfrak J^1 L)_{J, \mathrm{Re}}.
\]

\begin{remark}\label{rem:U(1)_action}
Derivation $\widehat j \mathbb 1$ of $\widehat L$ generates an $U(1)$-action on $\widehat L$ lifting the canonical one on $\widehat M$.
\end{remark}

\begin{lemma}
 Derivation $\mathcal L_{\widehat j \mathbb 1} - \widehat j{}^\dag$ of $\mathfrak J^1_{\mathbb R} \widehat L$ is a derivation of the jet algebroid $(\mathfrak J^1_{\mathbb R} \widehat L)_{\widehat J}$. Additionally, it generates a (necessarily free and proper) $U(1)$-action (necessarily by algebroid automorphisms) turning diagram
\[
\xymatrix{(\mathfrak J^1_{\mathbb R} \widehat L)_{4\widehat J} \ar[r] \ar[d] & \widehat M \ar[d]^-q \\
(\mathfrak J^1 L)_{J, \mathrm{Re}}  \ar[r] & M}
\] into a principal $U(1)$-bundle in the category of Lie algebroids, meaning that the $U(1)$-action is by Lie algebroid automorphisms.
\end{lemma}

\begin{proof}
Given any (non-necessarily Jacobi) biderivation $J$ of $\widehat L$, we denote by $b_J = [-,-]_J$ the associated skew-symmetric bracket (actually a biderivation) on $\mathfrak J^1_{\mathbb R} \widehat L$, obtained via formula:
\[
[\psi, \chi]_J = \mathcal L_{J^\sharp \psi} \chi - \mathcal L_{J^\sharp \chi} \psi - d_D \langle J, \psi \wedge \chi \rangle,
\]
$\psi, \chi \in \Gamma (\mathfrak J^1 L)$. The condition that $\mathcal L_{\widehat j \mathbb 1} - \widehat j{}^\dag$ is a Lie algebroid derivation is equivalent to
\begin{equation}\label{eq:SJ_der}
[\mathcal L_{\widehat j \mathbb 1} - \widehat j{}^\dag, b_{\widehat J}]^{SJ} = 0,
\end{equation}
where $[-,-]^{SJ}$ is the Schouten-Jacobi bracket on multiderivations of $\mathfrak J^1_{\mathbb R} \widehat L$. In order to prove (\ref{eq:SJ_der}) notice first that
\[
[\mathcal L_{\widehat j \mathbb 1}, b_{\widehat J}]^{SJ} = b_{[\widehat j \mathbb 1, \widehat J]^{SJ}} = - b_{\widehat J_{\widehat j}},
\]
where we used that $[\widehat j \mathbb 1, \widehat J]^{SJ} = - \widehat J_{\widehat j}$ (see \cite[Remark 59]{VW2016b}). So we have to prove that
\[
[\widehat j{}^\dag, b_{\widehat J}]^{SJ} = - b_{\widehat J_{\widehat j}}.
\]
But for every two sections $\psi, \chi $ of $J^1_{\mathbb R} \widehat L$ we have
\[
[\widehat j{}^\dag, b_{\widehat J}] (\psi, \chi) = \widehat j{}^\dag [\psi, \chi]_{\widehat J}-  [\widehat j{}^\dag \psi, \chi]_{\widehat J} -  [\psi, \widehat j{}^\dag \chi]_{\widehat J}
\]
and the latter agrees with $- [ \psi, \chi]_{\widehat J_{\widehat j}}$ precisely because $(\widehat L, \widehat J, \widehat j)$ is a Jacobi-Nijenhuis structure.

For the second part of the statement, let $\epsilon$ be a point in $L^\ast \smallsetminus 0$. We denote by $[\epsilon]$ its class in $\widehat M$. Recall that a section $\lambda$ of $L$ determines a section $\widehat \lambda$ of $\widehat L$ as in Section \ref{Sec:hom_trick}. The isomorphism $q^\ast \mathfrak J^1 L = \mathfrak J^1_{\mathbb R} \widehat L$ identifies pull-back sections with sections $\theta$ of $\mathfrak J^1_{\mathbb R} \widehat L$ such that
\begin{equation}\label{eq:pull}
(\mathcal L_{\widehat j \mathbb 1} - \widehat j{}^\dag) \theta = 0,
\end{equation}
(Proposition \ref{prop:hom_jet_compl}). Hence the values of these sections span all $ \mathfrak J^1_{\mathbb R} \widehat L$. This means that every point in $\mathfrak J^1_{\mathbb R} \widehat L$ is of the form $\theta_{[\epsilon]}$ for some $\theta$ such that (\ref{eq:pull}) and some $\epsilon$ in $L^\ast \smallsetminus 0$. For $\varphi \in U(1)$ we put $\varphi.\theta_{[\epsilon]} = \theta_{\varphi . [\epsilon]} = \theta_{[e^{i \varphi/2} \epsilon]}$. It is easy to see that $\varphi.\theta_{[\epsilon]}$ is well-defined, i.e.~it does only depend on $\theta_{[\epsilon]}$. So we have an $U(1)$-action on $\mathfrak J^1_{\mathbb R} \widehat L$. By construction it covers the canonical action on $\widehat M$. Finally, it immediately follows from (\ref{eq:pull}) and Remark \ref{rem:U(1)_action} that this action is generated by the linear vector field corresponding to derivation $\mathcal L_{\widehat j \mathbb 1} - \widehat j{}^\dag$. As the latter is a Lie algebroid derivation, the former is an action by Lie algebroid automorphisms.
\end{proof}

\begin{definition}
A holomorphic Jacobi manifold $(X, L, J)$ is \emph{integrable} if the holomorphic jet algebroid $(\mathfrak J^1 L)_{J}$ is integrable.
\end{definition}

We are finally in the position to prove our main result.

\begin{theorem} \label{main thm}
\quad
\begin{enumerate}
\item Let $(\mathcal G \rightrightarrows M, j_{\mathcal G}, H)$ be a holomorphic contact groupoid, and let $L_{\mathcal G}$ and $L$ be as in previous section. Then $M$ is equipped with a unique complex structure $j$ and a unique holomorphic Jacobi structure $(L, J)$, such that the natural projection $L_{\mathcal G} \to L$ is a holomorphic Jacobi map. The holomorphic Jacobi manifold $(M, j, L, J)$ is called the differentiation of $(\mathcal G, j_{\mathcal G}, H)$.
\item Conversely, let $(X = (M, j), L, J)$ be a holomorphic Jacobi manifold. Then $(X, L, J)$ is integrable iff $(\widehat M, \widehat L, 4\widehat J, \widehat j)$ (equivalently $(\widehat M, \widehat L, \widehat J, \widehat j)$) is an integrable Jacobi-Nijenhuis manifold, iff $(\widetilde X, \Omega, Z)$ is an integrable homogeneous complex symplectic manifold. In this case, let $(\mathcal G, j_{\mathcal G})$ be the source-simply connected complex groupoid integrating the holomorphic Lie algebroid $(\mathfrak J^1 L)_{J} \to X$. Then $(\mathcal G, j_{\mathcal G})$ is equipped with a unique multiplicative holomorphic contact structure $H$ such that $(X, L, J)$ differentiates $(\mathcal G, j_{\mathcal G}, H)$. The contact complex groupoid $(\mathcal G, j_{\mathcal G}, H)$ is called the intergation of $(X, L, J)$. Finally, the real contact groupoid $(\widehat G, \widehat H)$ from previous section integrates the real Jacobi manifold $(\widehat M, 4 \widehat J)$.
\end{enumerate}
Integration and differentiation establish a one-to-one correspondence between integrable holomorphic Jacobi manifolds and source-simply connected holomorphic contact groupoids. Shortly, holomorphic Jacobi manifolds integrate to holomorphic contact groupoids.
\end{theorem}

\begin{proof}
Let $(\mathcal G \rightrightarrows X, H)$ be a holomorphic contact groupoid, $X = (M, j)$. The (complex) symplectization $(\widetilde{\mathcal G} \rightrightarrows \widetilde X, \Omega, \mathcal Z)$ of $(\mathcal G, H)$ is a complex homogeneous symplectic groupoid. As usual, the holomorphic homogeneity vector field $\mathcal Z$ comes from a structure of principal $\mathbb C^\times$-bundle, in the category of Lie groupoids:
\[
\begin{array}{c}
\xymatrix{ \widetilde{\mathcal G} \ar@<-2pt>[r] \ar@<+2pt>[r] \ar[d] & \widetilde X \ar[d]\\
                \mathcal G \ar@<-2pt>[r] \ar@<+2pt>[r] & X}
\end{array},
\]
i.e.~$\mathbb C^\times$ acts freely and properly by Lie groupoid automorphisms. From
Proposition \ref{prop:hcsg}, $\widetilde X$ is canonically equipped with a holomorphic homogeneous Poisson structure $(\Pi, Z)$ where $Z$ is the Euler vector field. So $X$ is equipped with a unique holomorphic Jacobi structure $(L, J)$ with the desired properties, and $(\widetilde{\mathcal G}, \operatorname{Re} \Omega, 2 \operatorname{Re} Z)$ is the integration of the induced Jacobi-Nijenhuis manifold $(\widehat M, \widehat L, 4\widehat J, \widehat j)$.

Conversely, let $(X = (M, j), L, J)$ be a holomorphic Jacobi manifold, and let $(\widetilde X, \Pi, Z)$ be 
its Poissonization. The latter is a homogeneous holomorphic Poisson manifold. The holomorphic 
cotangent algebroid $(T^\ast \widetilde X)_\Pi$ is integrable iff its real Lie algebroid $(T^\ast \widetilde 
X)_{4 \operatorname{Re} \Pi}$ is integrable iff the jet algebroid $(\mathfrak J^1_{\mathbb R} \widehat 
L)_{4\widehat J}$ is integrable. Last claim follows from the fact that $(\widetilde X, 4\operatorname{Re} 
\Pi, 2 \operatorname{Re} Z)$ is the homogenization of $(\widehat M, 4\widehat J)$.
 Now suppose that $(\mathfrak J^1 L)_J$ is integrable. As $(\mathfrak J^1_{\mathbb R} \widehat 
 L)_{4\widehat J}$ is the action groupoid corresponding to the action of $(\mathfrak J^1 L)_{J, \operatorname{Re}}$ on $\widehat 
 M$, it follows that it is integrated by the action groupoid. Conversely, let $(\mathfrak J^1_{\mathbb R} 
 \widehat L)_{4\widehat J}$ integrate to the source symply connected Lie groupoid $\widehat{\mathcal G}
 $. Then the action of $U(1)$ on $(\mathfrak J^1_{\mathbb R} \widehat L)_{4\widehat J}$ by Lie algebroid 
 automorphisms, integrates to a free and proper action on $\widehat{\mathcal G}$ by Lie groupoid 
 automorphisms. As $(\mathfrak J^1 L)_{J, \operatorname{Re}}$ is the quotient Lie algebroid $(\mathfrak J^1_{\mathbb R} 
 \widehat L)_{4\widehat J} / U(1)$, then it is integrated by the quotient Lie groupoid $\widehat{\mathcal G} / 
 U(1)$.

Finally, suppose $(\mathfrak J^1 L)_J$ is integrable and let $\mathcal G$ be the source-symply connected complex Lie groupoid integrating it. The action of $(\mathfrak J^1 L)_J$ on $L$ integrates to an action of $\mathcal G$ on $L$, hence on $\widetilde X$. Consider the trivial side bundle complex LB-groupoid $(t^\ast L \rightrightarrows 0_X; \mathcal G \rightrightarrows X)$ corresponding to this action. The complex homogenization $\widetilde{ \mathcal G}$ of the holomorphic line bundle $t^\ast L \to \mathcal G$  is the source symply connected complex Lie groupoid integrating $(T^\ast X)_\Pi$. In particular, from Proposition \ref{prop:hcsg} again, $\widetilde{ \mathcal G}$ is equipped with a canonical multiplicative homogeneous symplectic structure $(\Omega, \mathcal Z)$, where $\mathcal Z$ is the Euler vector field. It follows that $\mathcal G$ is equipped with a unique multiplicative holomorphic contact structure $H$ with all the desired properties, in particular $T^{1,0}\mathcal G / H \simeq t^\ast L$.
\end{proof}

\color{black}

\begin{remark}
A \emph{contact realization} of a Jacobi manifold $(M, L, J)$ is a contact manifold $(P, H)$, together with a surjective and submersive Jacobi map $P \to M$ \cite{ZZ2006}. Similarly, one can define a \emph{holomorphic contact realization} of a holomorphic Jacobi manifold $(X, L, J)$ as a holomorphic contact manifold $(P, H)$ together with a surjective and submersive holomorphic Jacobi map $P \to M$. Theorem \ref{main thm} then immediately implies that integrable holomorphic Jacobi manifolds possess holomorphic contact realizations. It is now natural to wonder whether non-necessarily integrable holomorphic Jacobi manifolds possess holomorphic contact realizations or not. The analogous question for holomorphic Poisson manifolds, and (real) Jacobi manifolds, has been answered in the positive in terms of local Lie groupoids in \cite{LSX2009,BX2015} and \cite{CMS2017a,CMS2017b} respectively. We believe that a similar strategy as that of \cite{BX2015} can be used to prove that holomorphic Jacobi manifolds do always possess holomorphic contact realizations, and to provide an explicit construction. Probably, there is a shorter proof using the construction in \cite[Theorem A]{BX2015} and the homogenization scheme. But this would require a careful analysis of the interaction between that construction and the homogenization scheme in the holomorphic setting (Section \ref{Sec:complex_case}). Such an analysis goes beyond the scopes of the present paper and will be hopefully undertaken in a subsequent work.
\end{remark}

\color{black}

\emph{Acknowledgements.} We thank Alfonso G. Tortorella for providing a proof of Proposition \ref{prop:mult_Atiyah}. We also thank Damien Broka, Mathieu Sti\'enon, and Ping Xu for useful discussions. L.~V.~is member of the GNSAGA, INdAM.

\end{document}